\documentclass[12pt,a4paper]{amsart}

\usepackage{amssymb}

\newtheorem{theorem}{Theorem}[section]
\newtheorem{lemma}[theorem]{Lemma}
\newtheorem{corollary}[theorem]{Corollary}

\newtheorem{proposition}[theorem]{Proposition}

\theoremstyle{definition}
\newtheorem{definition}[theorem]{Definition}
\newtheorem{example}[theorem]{Example}

\theoremstyle{remark}
\newtheorem{remark}[theorem]{Remark}

\numberwithin{equation}{section}

\def\bc{\mathbb C}

\def\bbf{\mathbb{F}}
\def\mcf{\mathcal{F}}
\def\mct{\mathcal{T}}

\def\mand{\quad\mbox{and}\quad}

\def\bp{\begin{pmatrix}}
\def\ep{\end{pmatrix}}

\def\sdi{\,\mbox{$\rhd\kern-.55em<$}\,}
\def\edots{\mathinner{\mkern1mu\raise1pt\hbox{.}\mkern2mu\raise4pt\hbox{.}\mkern2mu\raise7pt\vbox{\kern7pt\hbox{.}}\mkern1mu}}

\begin{document}

\title[Orthogonal triple flag variety]{An example of orthogonal triple flag variety of finite type}

\author{Toshihiko MATSUKI}

\thanks{Supported by JSPS Grant-in-Aid for Scientific Research (C) \# 22540016.}

\address{Department of Mathematics\\
        Faculty of Science\\
        Kyoto University\\
        Kyoto 606-8502, Japan}
\email{matsuki@math.kyoto-u.ac.jp}
\date{}

\begin{abstract}
Let $G$ be the split special orthogonal group of degree $2n+1$ over a field $\bbf$ of ${\rm char}\,\bbf\ne 2$. Then we describe $G$-orbits on the triple flag varieties $G/P\times G/P\times G/P$ and $G/P\times G/P\times G/B$ with respect to the diagonal action of $G$ where $P$ is a maximal parabolic subgroup of $G$ of the shape $(n,1,n)$ and $B$ is a Borel subgroup. As by-products, we also describe ${\rm GL}_n$-orbits on $G/B$, $Q_{2n}$-orbits on the full flag variety of ${\rm GL}_{2n}$ where $Q_{2n}$ is the fixed-point subgroup in ${\rm Sp}_{2n}$ of a nonzero vector in $\bbf^{2n}$ and $1\times {\rm Sp}_{2n}$-orbits on the full flag variety of ${\rm GL}_{2n+1}$. In the same way, we can also solve the same problem for ${\rm SO}_{2n}$ where the maximal parabolic subgroup $P$ is of the shape $(n,n)$.
\end{abstract}

\maketitle

\section{Introduction}

Let $G$ be a reductive algebraic group over a field $\bbf$ and let $P_1,\ldots,P_k$ be parabolic subgroups of $G$. Then we consider the diagonal action of $G$ on the multiple flag variety
$$\mathcal{M}=(G/P_1)\times\cdots\times(G/P_k).$$
We say $\mathcal{M}$ is of finite type if it has finite number of $G$-orbits when the field $\bbf$ is infinite.

In \cite{MWZ1}, Magyar, Weyman and Zelevinsky classified multiple flag varieties of finite type for ${\rm GL}_n(\bbf)$ with an arbitrary algebraically closed field $\bbf$ and described their orbit decompositions using quiver theory. In \cite{MWZ2}, they also solved the same problem for ${\rm Sp}_{2n}(\bbf)$. 

Consider a triple flag variety $\mathcal{M}=(G/P_1)\times (G/P_2)\times (G/P_3)$ and note that $G$-orbit decompostion on $\mathcal{M}$ is naturally identified with $P_3$-orbit decomposition on the double flag variety $\mathcal{D}=(G/P_1)\times (G/P_2)$. Littelmann (\cite{L}) classified double flag varieties $\mathcal{D}$ with open $B$-orbits for simple algebraic groups $G$. Here $P_1$ and $P_2$ are maximal parabolic subgroups of $G$ and $B$ a Borel subgroup of $G$. Suppose that $\mathcal{D}$ has an open $B$-orbit and that $\bbf$ is an algebraically closed field of ${\rm char}\,\bbf=0$. Then it follows from the theorem by
Brion (\cite{B}) and Vinberg (\cite{V}) that $|B\backslash \mathcal{D}|$ is finite.

We can see there are many open problems in this subject. One of them is an explicit description of orbit decomposition for each triple flag variety classified in \cite{L} (Table I). In this paper, we solve this problem for some typical orthogonal triple flag variety. It is interesting that we can describe orbits in our example over an arbitrary field of ${\rm char}\,\bbf\ne 2$ and so we can also compute the number of elements in each orbit when $\bbf$ is a finite field.

Let $\bbf$ be an arbitrary commutative field of ${\rm char}\,\bbf\ne 2$. Let $(\ ,\ )$ denote the symmetric bilinear form on $\bbf^{2n+1}$ defined by
$$(e_i,e_j)=\delta_{i,2n-i+2}$$
for $i,j=1,\ldots,2n+1$ where $e_1,\ldots,e_{2n+1}$ is the canonical basis of $\bbf^{2n+1}$. Define the special orthogonal group
$$G=\{g\in {\rm SL}_{2n+1}(\bbf) \mid (gu,gv)=(u,v)\mbox{ for all }u,v\in\bbf^{2n+1}\}$$
with respect to this form. Let us write $G={\rm SO}_{2n+1}(\bbf)$ in this paper. Let $M$ denote the variety consisting of all the maximal isotropic subspaces in $\bbf^{2n+1}$. Here a subspace $V$ in $\bbf^{2n+1}$ is called a maximal isotropic subspace if $\dim V=n$ and $(V,V)=\{0\}$. Then $M$ is a homogeneous space of $G$ and hence it is written as
$M\cong G/P$
where $P=\{g\in G\mid gU_0=U_0\}$ with $U_0=\bbf e_1\oplus\cdots\oplus \bbf e_n\in M$ is a maximal parabolic subgroup of $G$.

Let $M_0=\{V_1\subset\cdots\subset V_n\mid (V_n,V_n)=\{0\}\}\cong G/B$ denote the full flag variety of $G$. Here $B$ is the isotropy subgroup of the canonical full flag $\bbf e_1\subset\bbf e_1\oplus \bbf e_2\subset\cdots\subset \bbf e_1\oplus\cdots\oplus \bbf e_n$ in $M_0$ which is called a Borel subgroup of $G$. In this paper, we will describe $G$-orbits on $\mathcal{T}=M\times M\times M$ and $\mathcal{T}_0=M\times M\times M_0$ with respect to the diagonal action.

In the same way, we can also solve the problem for ${\rm SO}_{2n}(\bbf)$ (Section 1.5).

\begin{remark} \ (i) Similar problems were studied in \cite{KS}, \cite{FMS} and \cite{CN}. In particular, \cite{KS} (p.492) described ${\rm Sp}_{2n}(\mathbb{R})$-orbits on the variety consisting  of triples of Lagrangian subspaces in a real symplectic vector space. For each orbit in this decomposition, there corresponds a symmetric bilinear form and the ``Maslov index'' is naturally defined. So it is natural that there appear alternating forms in our results on ${\rm SO}_{2n+1}(\bbf)$-orbit decompositions of $\mathcal{T}$ (Theorem \ref{th1.4}) and $\mathcal{T}_0$ (Theorem \ref{th1.6}).

(ii) We may consider the action of the orthogonal group
$$\widetilde{G}={\rm O}_{2n+1}(\bbf)=\{g\in {\rm GL}_{2n+1}(\bbf) \mid (gu,gv)=(u,v)\mbox{ for all }u,v\in\bbf^{2n+1}\}$$
on $M$ and $M_0$. But since $\widetilde{G}=G\sqcup \{-g\mid g\in G\}$ and since $-I_{2n+1}$ acts trivially on $M$ and $M_0$, the $\widetilde{G}$-orbits are the same as the $G$-orbits.

(iii) The triple flag variety $\mathcal{T}_0$ has the maximum dimension among the triple flag varieties of ${\rm SO}_{2n+1}(\bbf)$ of finite type since
$$\dim \mathcal{T}_0={n(n+1)\over 2}+{n(n+1)\over 2}+n^2=n(2n+1)=\dim {\rm SO}_{2n+1}(\bbf).$$
\end{remark}

\subsection{$G$-orbits on $\mct=M\times M\times M$}

For $d=0,\ldots,n$, define $U_d=\bbf e_1\oplus\cdots\oplus\bbf e_{n-d}\oplus \bbf e_{n+2}\oplus\cdots\oplus\bbf e_{n+d+1}\in M$.
For a partition $n=a+b+c_++c_0+c_-$ of $n$ with nonnegative integers $a,b,c_+,c_0$ and $c_-$, define subspaces
\begin{align*}
U_{(\alpha)} & =\bbf e_1\oplus\cdots\oplus\bbf e_a, \quad
U_{(\beta)} =\bbf e_{2n-a-b+2}\oplus\cdots\oplus\bbf e_{2n-a+1}, \\
U_{(+)} & =\bbf e_{a+b+1}\oplus\cdots\oplus\bbf e_{a+b+c_+}, \quad
U_{(-)} =\bbf e_{n+2}\oplus\cdots\oplus\bbf e_{n+c_-+1}, \\
U_{(0)} & =\bbf e_{a+b+c_++1}\oplus\cdots\oplus\bbf e_{a+b+c_++c_0} \oplus\bbf e_{n+c_-+2}\oplus\cdots\oplus\bbf e_{n+c_-+c_0+1} \oplus\bbf e_{n+1} 
\end{align*}
of $\bbf^{2n+1}$. Write $W_{(0)}=U_{(\alpha)}\oplus U_{(\beta)}\oplus U_{(+)}\oplus U_{(-)},\ k_+=a+b+c_+$ and $k_-=n+c_-+1$. If $c_0=2c_1-1$ is odd, then we define an element of $M$ by
\begin{align*}
V(a,b,c_+,c_-)_{\rm odd} & =W_{(0)}\oplus \left(\bigoplus_{i=1}^{c_1-1}\bbf (e_{k_++i}+e_{k_-+i})\right) \oplus \left(\bigoplus_{i=c_1+1}^{c_0}\bbf (e_{k_++i}-e_{k_-+i})\right) \\
& \qquad \oplus \bbf (e_{k_++c_1}-{1\over 2}e_{k_-+c_1}+e_{n+1}).
\end{align*}
If $c_0=2c_1$ is even, then we define an element of $M$ by
\begin{align*}
V(a,b,c_+,c_-)_{\rm even}^0 & =W_{(0)}\oplus \left(\bigoplus_{i=1}^{c_1}\bbf (e_{k_++i}+e_{k_-+i})\right) \oplus \left(\bigoplus_{i=c_1+1}^{c_0}\bbf (e_{k_++i}-e_{k_-+i})\right).
\end{align*}
If $c_0=2c_1$ is even and positive, then we also define
\begin{align*}
V(a,b,c_+,c_-)_{\rm even}^1 & =W_{(0)}\oplus \left(\bigoplus_{i=1}^{c_1}\bbf (e_{k_++i}+e_{k_-+i})\right) \oplus \left(\bigoplus_{i=c_1+1}^{c_0-1}\bbf (e_{k_++i}-e_{k_-+i})\right) \\
& \qquad \oplus \bbf (e_{k_++c_0}-e_{k_-+c_0}-{1\over 2}e_{k_-+1}+e_{n+1}) \in M.
\end{align*}

\begin{theorem} \ Let $t=(V_{(1)},V_{(2)},V_{(3)})$ be an element of $\mct=M\times M\times M$. Define
\begin{align*}
a & =a(t)=\dim (V_{(1)}\cap V_{(2)}\cap V_{(3)}),\ b=b(t)=\dim (V_{(1)}\cap V_{(2)}) -a, \\
c_+ & =c_+(t)=\dim (V_{(1)}\cap V_{(3)}) -a,\ c_-=c_-(t)=\dim (V_{(2)}\cap V_{(3)}) -a, \\
c_0 & =c_0(t)=n-a-b-c_+-c_- \\
\mand \varepsilon & =\varepsilon(t)=\dim(V_{(1)}+V_{(2)}+V_{(3)})+\dim(V_{(1)}\cap V_{(2)}\cap V_{(3)})-2n\in \{0,1\}.\end{align*}

{\rm (i)} If $c_0$ is odd, then $\varepsilon=1$ and $t\in G(U_0,U_{n-a-b},V(a,b,c_+,c_-)_{\rm odd})$.

{\rm (ii)} If $c_0=0$, then $\varepsilon=0$ and $t\in G(U_0,U_{n-a-b},V(a,b,c_+,c_-)_{\rm even}^0)$.

{\rm (iii)} If $c_0$ is even and positive, then $t\in G(U_0,U_{n-a-b},V(a,b,c_+,c_-)_{\rm even}^\varepsilon)$ with $\varepsilon=0\mbox{ or }1$.
\label{th1.2}
\end{theorem}

\begin{corollary} \ $\displaystyle{|G\backslash \mct|=\sum_{k=0}^n \eta_k {n-k+3\choose 3}}$ where
$\displaystyle{\eta_k=\begin{cases} 1 & \text{if $k=0,1,3,5,\ldots$}, \\ 2 & \text{if $k=2,4,6,\ldots$}. \end{cases}}$
\end{corollary}

For $n=1,2,3,4$, the number of orbits $|G\backslash \mct|$ is as follows.

\bigskip
\centerline{
\vbox{\offinterlineskip
\hrule
\halign{&\vrule#&\strut\quad$\hfil#\hfil$\quad\cr
height2pt&\omit&&\omit&&\omit&&\omit&&\omit&\cr
& n && 1 && 2 && 3 && 4 &\cr
height2pt&\omit&&\omit&&\omit&&\omit&&\omit&\cr
\noalign{\hrule}
height4pt&\omit&&\omit&&\omit&&\omit&&\omit&\cr
& |G\backslash \mct| && 5 && 16 && 39 && 81 &\cr
height4pt&\omit&&\omit&&\omit&&\omit&&\omit&\cr}
\hrule}
}

\begin{theorem} \ When $\bbf$ is the finite field $\bbf_r$ with $r$ elements, the number of elements in the $G$-orbit $Gt$ is
$$|Gt|=|M|{r^{(n-a)(n-a+1)/2}[r]_n\over [r]_a[r]_b[r]_{c_+}[r]_{c_-}[r]_{c_0}}\psi_{c_0}^\varepsilon(r).$$
Here $[r]_m$ is the $r$-factorial number $(r+1)(r^2+r+1)\cdots(r^{m-1}+r^{m-2}+\cdots+1)$ and
\begin{align*}
\psi_{2k}^0(r) & =\psi_{2k-1}^1(r)={\psi_{2k}^1(r)\over r^{2k}-1}
=r^{k(k-1)}(r-1)(r^3-1)\cdots(r^{2k-1}-1).
\end{align*}
\label{th1.4}
\end{theorem}

\begin{remark} (c.f. Proposition \ref{prop1.5})\ $\psi_{c_0}^\varepsilon(r)=|{\rm GL}_{c_0}(\bbf_r)/H_{c_0}^\varepsilon|$ where
$$H_{c_0}^\varepsilon=\begin{cases} 1\times {\rm Sp}_{c_0-1}(\bbf_r) & \text{if $c_0$ is odd}, \\
{\rm Sp}_{c_0}(\bbf_r) & \text{if $c_0$ is even and $\varepsilon=0$}, \\
Q_{c_0}=\{g\in {\rm Sp}_{c_0}(\bbf_r)\mid gv=v\} & \text{if $c_0$ is even and $\varepsilon=1$.}
\end{cases}$$
($v$ is a nonzero element in $\bbf_r^{c_0}$.)
\end{remark}

\subsection{$G$-orbits on $\mct_0=M\times M\times M_0$}

By Theorem \ref{th1.2}, we may fix a $t=(U_0,U_d,V)$ with $V\in M$ of the form
$$V=V(a,b,c_+,c_-)_{\rm odd},\ V(a,b,c_+,c_-)_{\rm even}^0\mbox{ or }V(a,b,c_+,c_-)_{\rm even}^1$$
where $d=n-a-b$. Let $M_0(V)$ denote the subvariety of $M_0$ consisting of full flags $\mcf: V_1\subset\cdots\subset V_n$ satisfying $V_n=V$. Let $\pi: \mct_0\to \mct$ be the projection. Then the fiber $\pi^{-1}(t)$ at $t$ is naturally identified with $M_0(V)$. Since the isotropy subgroup at $t$ is $R(t)=P\cap P_{U_d}\cap P_V$. We have only to describe $R(t)$-orbits on $M_0(V)$.

\begin{definition} \ A full flag $\mcf: V_1\subset\cdots\subset V_n$ in $M_0(V)$ is called standard if
$$V_i =(V_i\cap U_{(\alpha)})\oplus (V_i\cap U_{(\beta)})\oplus (V_i\cap (U_{(+)}\oplus U_{(-)}))\oplus (V_i\cap U_{(0)}),$$
$V_i\cap U_{(\alpha)} =\bbf e_1\oplus\cdots\oplus \bbf e_{a_i(\mcf)}$
and $V_i\cap U_{(\beta)} =\bbf e_{2n-a-b+2}\oplus\cdots\oplus \bbf e_{2n-a-b+1+b_i(\mcf)}$ for all $i=1,\ldots,n$ where $a_i(\mcf)=\dim (V_i\cap U_{(\alpha)})$ and $b_i(\mcf)=\dim (V_i\cap U_{(\beta)})$.
\label{def1.6}
\end{definition}

For a standard full flag $\mcf: V_1\subset\cdots\subset V_n$, write $c_i(\mcf)=\dim(V_i\cap (U_{(+)}\oplus U_{(-)}))$ and $d_i(\mcf)=\dim(V_i\cap U_{(0)})$. Define subsets
\begin{align*}
I_{(\alpha)} & =\{\alpha_1,\ldots,\alpha_a\} =\{i\in I\mid a_i(\mcf)=a_{i-1}(\mcf)+1\}, \\
I_{(\beta)} & =\{\beta_1,\ldots,\beta_b\} =\{i\in I\mid b_i(\mcf)=b_{i-1}(\mcf)+1\}, \\
I_{(\gamma)} & =\{\gamma_1,\ldots,\gamma_c\} =\{i\in I\mid c_i(\mcf)=c_{i-1}(\mcf)+1\}, \\
I_{(\delta)} & =\{\delta_1,\ldots,\delta_{c_0}\} =\{i\in I\mid d_i(\mcf)=d_{i-1}(\mcf)+1\}
\end{align*}
of $I=\{1,\ldots,n\}$ where $c=c_++c_-$ and $\alpha_1<\cdots<\alpha_a,\ \beta_1<\cdots<\beta_b,\ \gamma_1<\cdots<\gamma_c,\ \delta_1<\cdots<\delta_{c_0}$. Let $\tau(\mcf)$ denote the permutation
$$\tau(\mcf): (1\,2\cdots n)\mapsto (\alpha_1\cdots\alpha_a\gamma_1\cdots\gamma_c\delta_1\cdots\delta_{c_0}\beta_1\cdots\beta_b)$$
of $I$ and $\ell(\tau(\mcf))$ the inversion number of $\tau(\mcf)$.

For $X\in {\rm GL}_n(\bbf)$, write
$$h[X]=\bp X & 0 & 0 \\ 0 & 1 & 0 \\ 0 & 0 & J\,{}^tX^{-1}J \ep\mbox{ with }J=J_n=\bp 0 && 1 \\ & \edots & \\ 1 && 0 \ep.$$
For $A\in {\rm GL}_{c_+}(\bbf),\ B\in {\rm GL}_{c_0}(\bbf)$ and $C\in {\rm GL}_{c_-}(\bbf)$, define an element
$$\ell(A,B,C)=h\left[\bp I_{a+b} & 0 & 0 & 0 \\ 0 & A & 0 & 0 \\ 0 & 0 & B & 0 \\ 0 & 0 & 0 & C \ep\right]$$
of $G$. Let $L_+,\ L_0,\ L_-,\ L$ and $L_V$ denote the subgroups of $G$ defined by
\begin{align*}
L_+ & =\{\ell(A,I_{c_0},I_{c_-})\mid A\in {\rm GL}_{c_+}(\bbf)\}, \\
L_0 & =\{\ell(I_{c_+},B,I_{c_-})\mid B\in {\rm GL}_{c_0}(\bbf)\}, \\
L_- & =\{\ell(I_{c_+},I_{c_0},C)\mid C\in {\rm GL}_{c_-}(\bbf)\},
\end{align*}
$L=L_+\times L_0\times L_-$ and $L_V=\{\ell\in L\mid \ell V=V\}$, respectively.

\begin{proposition} \ {\rm (i)} $L_V=L_+\times (L_V\cap L_0)\times L_-$.

{\rm (ii)} $V=V(a,b,c_+,c_-)_{\rm odd}\Longrightarrow L_V\cap L_0\cong 1\times {\rm Sp}_{c_0-1}(\bbf)$,

\quad $V=V(a,b,c_+,c_-)_{\rm even}^0\Longrightarrow L_V\cap L_0\cong {\rm Sp}_{c_0}(\bbf)$,

\quad $V=V(a,b,c_+,c_-)_{\rm even}^1\Longrightarrow L_V\cap L_0\cong Q_{c_0}$.

\noindent Here $Q_{c_0}=\{g\in {\rm Sp}_{c_0}(\bbf)\mid gv=v\}$ with some $v\in\bbf^{c_0}-\{0\}$.
\label{prop1.5}
\end{proposition}

\begin{theorem} \ {\rm (i)} For every full flag $\mcf$ in $M_0(V)$, there exists a $g\in R(t)=P\cap P_{U_d}\cap P_V$ such that $g\mcf$ is standard.

{\rm (ii)} Let $\mcf$ and $\mcf'$ be two standard full flags such that $g\mcf=\mcf'$ for some $g\in R(t)$. Then there exists a $g_L\in L_V$ such that $g_L\mcf=\mcf'$.

{\rm (iii)} If $\bbf=\bbf_r$, then $|R(t)\mcf|=[r]_a[r]_br^{\ell(\tau(\mcf))}|L_V\mcf|$ for each standard full flag $\mcf$ in $M_0(V)$. 
\label{th1.6}
\end{theorem}

\subsection{Orbits on ${\rm GL}_n(\bbf)/B$}

By Proposition \ref{prop1.5} and Theorem \ref{th1.6}, our problem is reduced to the orbit decompositions on the full flag variety of ${\rm GL}_n(\bbf)$ with respect to the following four kinds of subgroups $H$ of ${\rm GL}_n(\bbf)$:

\bigskip
(A) $H={\rm GL}_{m_+}(\bbf)\times {\rm GL}_{m_-}(\bbf)$ where $m_++m_-=n$,

(B) $H={\rm Sp}_n(\bbf)$ for even $n$,

(C) $H=Q_n$ for even $n$,

(D) $H=1\times {\rm Sp}_{n-1}(\bbf)$ for odd $n$.
\bigskip

When $\bbf=\bc$, the subgroups $H$ in (A) and (B) are symmetric subgroups of ${\rm GL}_n(\bc)$ and the orbit structures were described in \cite{M1} and \cite{R}. We also have symbolic description of orbits in \cite{MO}. 


We will solve Problems (B), (C) and (D) in Section 3. We will also give a proof for the Problem (A) in the appendix. We don't need the assumption ${\rm char}\,\bbf \ne 2$ for these problems.

We can express ${\rm GL}_{m_+}(\bbf)\times {\rm GL}_{m_-}(\bbf)$-orbits on $M={\rm GL}_n(\bbf)/B$ by ``+$-$ab-symbols''. For example, when $m_+=m_-=2$, the orbit structure is as follows (Fig.7 in \cite{MO}).

\setlength{\unitlength}{1mm}

\begin{picture}(140,105)(0,-10)
\put(20,85){\makebox(0,0){++$-$$-$}}
\put(40,85){\makebox(0,0){+$-$+$-$}}
\put(60,85){\makebox(0,0){+$-$$-$+}}
\put(80,85){\makebox(0,0){$-$++$-$}}
\put(100,85){\makebox(0,0){$-$+$-$+}}
\put(120,85){\makebox(0,0){$-$$-$++}}
\put(20,65){\makebox(0,0){+aa$-$}}
\put(40,65){\makebox(0,0){+$-$aa}}
\put(60,65){\makebox(0,0){aa+$-$}}
\put(80,65){\makebox(0,0){aa$-$+}}
\put(100,65){\makebox(0,0){$-$+aa}}
\put(120,65){\makebox(0,0){$-$aa+}}
\put(30,45){\makebox(0,0){+a$-$a}}
\put(50,45){\makebox(0,0){a+a$-$}}
\put(70,45){\makebox(0,0){aabb}}
\put(90,45){\makebox(0,0){a$-$a+}}
\put(110,45){\makebox(0,0){$-$a+a}}
\put(40,25){\makebox(0,0){a+$-$a}}
\put(70,25){\makebox(0,0){abab}}
\put(100,25){\makebox(0,0){a$-$+a}}
\put(70,5){\makebox(0,0){abba}}
\put(20,82){\vector(0,-1){14}} \put(17,75){\makebox(0,0){$2$}}
\put(37,82){\vector(-1,-1){14}} \put(27,76){\makebox(0,0){$2$}}
\put(40,82){\vector(0,-1){14}} \put(37,75){\makebox(0,0){$3$}}
\put(43,82){\vector(1,-1){14}} \put(44,78){\makebox(0,0){$1$}}
\put(57,82){\vector(-1,-1){14}} \put(56,78){\makebox(0,0){$3$}}
\put(63,82){\vector(1,-1){14}} \put(64,78){\makebox(0,0){$1$}}
\put(77,82){\vector(-1,-1){14}} \put(76,78){\makebox(0,0){$1$}}
\put(83,82){\vector(1,-1){14}} \put(84,78){\makebox(0,0){$3$}}
\put(97,82){\vector(-1,-1){14}} \put(96,78){\makebox(0,0){$1$}}
\put(100,82){\vector(0,-1){14}} \put(103,75){\makebox(0,0){$3$}}
\put(103,82){\vector(1,-1){14}} \put(113,76){\makebox(0,0){$2$}}
\put(120,82){\vector(0,-1){14}} \put(123,75){\makebox(0,0){$2$}}
\put(21.5,62){\vector(1,-2){7}} \put(22,55){\makebox(0,0){$3$}}
\put(61.5,62){\vector(1,-2){7}} \put(67,57){\makebox(0,0){$3$}}
\put(78.5,62){\vector(-1,-2){7}} \put(73,57){\makebox(0,0){$3$}}
\put(118.5,62){\vector(-1,-2){7}} \put(118,55){\makebox(0,0){$3$}}
\put(24.5,62){\vector(3,-2){21}} \put(30,62){\makebox(0,0){$1$}}
\put(44.5,62){\vector(3,-2){21}} \put(50,62){\makebox(0,0){$1$}}
\put(95.5,62){\vector(-3,-2){21}} \put(90,62){\makebox(0,0){$1$}}
\put(115.5,62){\vector(-3,-2){21}} \put(110,62){\makebox(0,0){$1$}}
\put(38.5,62){\vector(-1,-2){7}} \put(39,59){\makebox(0,0){$2$}}
\put(58.5,62){\vector(-1,-2){7}} \put(59,59){\makebox(0,0){$2$}}
\put(81.5,62){\vector(1,-2){7}} \put(81,59){\makebox(0,0){$2$}}
\put(101.5,62){\vector(1,-2){7}} \put(101,59){\makebox(0,0){$2$}}
\put(31.5,42){\vector(1,-2){7}} \put(32,35){\makebox(0,0){$1$}}
\put(48.5,42){\vector(-1,-2){7}} \put(48,35){\makebox(0,0){$3$}}
\put(70,42){\vector(0,-1){14}} \put(73,35){\makebox(0,0){$2$}}
\put(91.5,42){\vector(1,-2){7}} \put(92,35){\makebox(0,0){$3$}}
\put(108.5,42){\vector(-1,-2){7}} \put(108,35){\makebox(0,0){$1$}}
\put(44.5,22){\vector(3,-2){21}} \put(52,14){\makebox(0,0){$2$}}
\put(95.5,22){\vector(-3,-2){21}} \put(88,14){\makebox(0,0){$2$}}
\put(68.5,22){\vector(0,-1){14}} \put(66,16){\makebox(0,0){$1$}}
\put(71.5,22){\vector(0,-1){14}} \put(74,16){\makebox(0,0){$3$}}
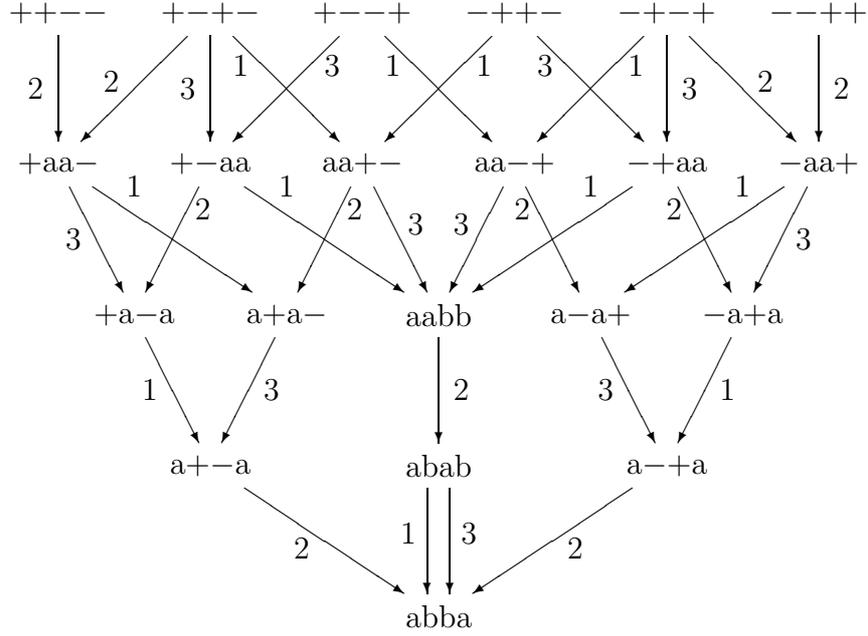
\put(70,-5){\makebox(0,0){Fig.1. ${\rm GL}_2(\bbf)\times {\rm GL}_2(\bbf)\backslash {\rm GL}_4(\bbf)/B$}}
\end{picture}


\bigskip
\noindent Notation: For $i=1,\ldots,n-1$, we can consider the partial flag variety
$$M_i=\{V_1\subset\cdots\subset V_{i-1}\subset V_{i+1}\subset\cdots\subset V_{n-1} \mid \dim V_j=j\}$$
and the canonical projection $p_i:M\to M_i$. For two $H$-orbits $S_1$ and $S_2$ in $M$, we write $S_1\xrightarrow{i} S_2$ when $p_i(S_1)=p_i(S_2)$ and $\dim S_1+1=\dim S_2$. (Remark: In our setting, every orbit in $M$ is defined by linear equations. So we can define ``dimension'' of each orbit over an arbitrary field $\bbf$. When $\bbf=\bbf_r$, the number $|S|$ of points in a orbit $S$ is a polynomial of $r$ and $\dim S$ is the degree of the polynomial.) This notation will be also used in Problems (B), (C) and (D).

\bigskip
Problem (B) is solved as follows. 
Let $G={\rm GL}_{2n}(\bbf)$ and define a nondegenerate alternating form $\langle\ ,\ \rangle$ on $\bbf^{2n}$ by
$$\langle e_i,e_j\rangle=\begin{cases} \delta_{i,2n+1-j} & \text{for $i=1,\ldots,n$}, \\
-\delta_{i,2n+1-j} & \text{for $i=n+1,\ldots,2n$}.
\end{cases}$$
For a subspace $V$ in $\bbf^{2n}$, let $V^\perp=\{v\in V\mid \langle v,V\rangle =\{0\}\}$ denote the orthogonal space for $V$. Define $H=\{g\in G\mid \langle gu,gv\rangle=\langle u,v\rangle\mbox{ for all }u,v\in\bbf^{2n}\}$. Then $H$ is isomorphic to ${\rm Sp}_{2n}(\bbf)$. Let
$\mcf: V_1\subset V_2\subset\cdots\subset V_{2n-1}$
be a full flag in $\bbf^{2n}$. Define $d_{i,j}=d_{i,j}(\mcf)=\dim(V_i\cap V_j^\perp)$ for $i,j=0,\ldots,2n$ and $c_{i,j}=c_{i,j}(\mcf)=d_{i,j-1}-d_{i,j}-d_{i-1,j-1}+d_{i-1,j}$ for $i,j=1,\ldots,2n$. We prove the following propositions in Section 3.2.

\begin{proposition} \ The matrix $\{c_{i,j}\}_{i=1,\ldots,2n}^{j=1,\ldots,2n}$ is a symmetric permutation matrix such that $c_{i,i}=0$ for $i=1,\ldots,2n$.
\label{prop1.9}
\end{proposition}

Let $\tau=\tau(\mcf)$ be the permutation of $I=\{1,\ldots,2n\}$ corresponding to $\{c_{i,j}\}=\{c_{i,j}(\mcf)\}$. Then $\tau$ is expressed as
$\tau=(i_1\ j_1)\cdots(i_n\ j_n)$
with transpositions $(i_1\ j_1),\ldots,$ $(i_n\ j_n)$. We may assume
$$i_t<j_t\mbox{ for }t=1,\ldots,n\mand i_1<i_2<\cdots<i_n.$$
(Hence $i_1=1$.) Define another permutation $\sigma$ by
$$\sigma=\sigma(\mcf):(1\,2\cdots 2n)\mapsto (i_1j_1i_2j_2\cdots i_nj_n).$$ 
Let $\ell(\sigma)$ denote the inversion number
$\ell(\sigma)=|\{(i,j)\mid i<j\mbox{ and }\sigma(i)>\sigma(j)\}|$. (Remark: We can prove $\ell(\tau)=n+2\ell(\sigma)$.)

\begin{proposition} \ {\rm (i)} There exists a basis $v_1,\ldots,v_{2n}$ of $\bbf^{2n}$ satisfying {\rm (a)} and {\rm (b)}$:$

\indent\indent {\rm (a)} $V_i=\bbf v_1\oplus\cdots\oplus \bbf v_i$ for $i=1,\ldots,2n$.

\indent\indent {\rm (b)} $\langle v_i,v_j\rangle=c_{i,j}$ for $i<j$.

{\rm (ii)} If $\bbf=\bbf_r$, then the number of bases satisfying the properties {\rm (a)} and {\rm (b)} in {\rm (i)} is $(r-1)^nr^{n+\ell(\sigma)}$.
\label{prop1.10}
\end{proposition}

Let $C_{2n}$ denote the set of symmetric permutation matrices $\{c_{i,j}\}$ of degree $2n$ such that $c_{i,i}=0$ for $i=1,\ldots,2n$. By Proposition \ref{prop1.9} and Proposition \ref{prop1.10}, we have:

\begin{corollary} \ {\rm (i)} There exists a one-to-one correspondence between $C_{2n}$ and ${\rm Sp}_{2n}(\bbf)\backslash {\rm GL}_{2n}(\bbf)/B$.

{\rm (ii)} $\displaystyle{|{\rm Sp}_{2n}(\bbf)\backslash {\rm GL}_{2n}(\bbf)/B| =(2n-1)(2n-3)\cdots 1={(2n)!\over 2^nn!}}$.

{\rm (iii)} If $\bbf=\bbf_r$, then
$$|H\mcf|={|{\rm Sp}_{2n}(\bbf_r)|\over (r-1)^nr^{n+\ell(\sigma(\mcf))}}
={(r^2-1)(r^4-1)\cdots(r^{2n}-1)\over (r-1)^n} r^{n^2-n-\ell(\sigma(\mcf))}.$$
\label{cor1.11}
\end{corollary}

By this result, we can describe ${\rm Sp}_{2n}(\bbf)\backslash {\rm GL}_{2n}(\bbf)/B$ by the ``AB-symbols''. For $n=2,3$, the orbit structure is as in Fig.2 and Fig.3 (Fig.3 and Fig.4 in \cite{MO}). For example, the symbol ABBA implies the ${\rm Sp}_4$-orbit of the flag $\mcf$ such that $c_{1,4}=c_{2,3}=1$.

Problem (C) is solved as follows. Retain the notations for Problem (B). Define a subgroup $Q_{2n}=\{g\in H\mid ge_{2n}=e_{2n}\}$ of $H\cong {\rm Sp}_{2n}(\bbf)$. Write $W=(\bbf e_{2n})^\perp=\bbf e_2\oplus\cdots\oplus \bbf e_{2n}$. Let $S$ denote the subset of $I\times I$ defined by
$S=S(\mcf)=\{(i,j)\mid V_i\cap V_{j-1}^\perp\not\subset W\}$
and define a subset
$$S_0=S_0(\mcf)=\{(i,j)\in S\mid V_i\cap V_j^\perp\subset W\mbox{ and }V_{i-1}\cap V_{j-1}^\perp\subset W\}$$
of $S$. Then the following results will be proved in Section 3.3.

\begin{picture}(140,150)
\put(30,100){\makebox(0,0){ABBA}}
\put(29,97){\vector(0,-1){14}} \put(26,90){\makebox(0,0){$1$}}
\put(31,97){\vector(0,-1){14}} \put(34,90){\makebox(0,0){$3$}}
\put(30,80){\makebox(0,0){ABAB}}
\put(30,77){\vector(0,-1){14}} \put(27,70){\makebox(0,0){$2$}}
\put(30,60){\makebox(0,0){AABB}}
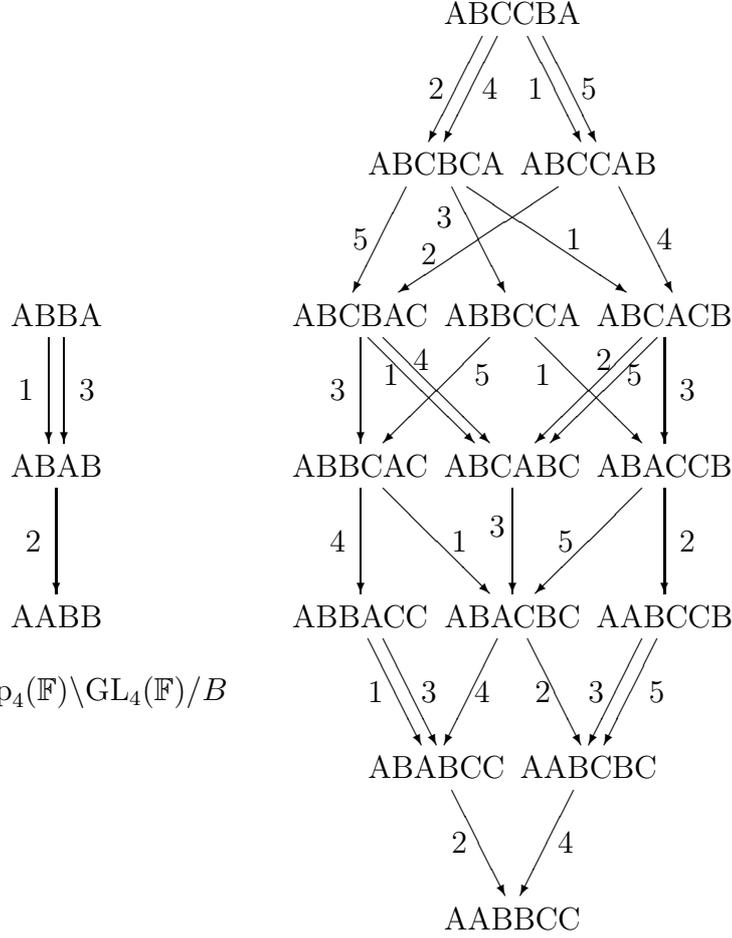
\put(30,50){\makebox(0,0){Fig.2. ${\rm Sp}_4(\bbf)\backslash {\rm GL}_4(\bbf)/B$}}
\put(90,140){\makebox(0,0){ABCCBA}}
\put(86,137){\vector(-1,-2){7}} \put(80,130){\makebox(0,0){$2$}}
\put(88,137){\vector(-1,-2){7}} \put(87,130){\makebox(0,0){$4$}}
\put(92,137){\vector(1,-2){7}} \put(93,130){\makebox(0,0){$1$}}
\put(94,137){\vector(1,-2){7}} \put(100,130){\makebox(0,0){$5$}}
\put(80,120){\makebox(0,0){ABCBCA}}
\put(100,120){\makebox(0,0){ABCCAB}}
\put(76,117){\vector(-1,-2){7}} \put(70,110){\makebox(0,0){$5$}}
\put(82,117){\vector(1,-2){7}} \put(81,113){\makebox(0,0){$3$}}
\put(84,117){\vector(3,-2){21}} \put(98,110){\makebox(0,0){$1$}}
\put(104,117){\vector(1,-2){7}} \put(110,110){\makebox(0,0){$4$}}
\put(96,117){\vector(-3,-2){21}} \put(79,108){\makebox(0,0){$2$}}
\put(70,100){\makebox(0,0){ABCBAC}}
\put(90,100){\makebox(0,0){ABBCCA}}
\put(110,100){\makebox(0,0){ABCACB}}
\put(70,97){\vector(0,-1){14}} \put(67,90){\makebox(0,0){$3$}}
\put(110,97){\vector(0,-1){14}} \put(113,90){\makebox(0,0){$3$}}
\put(71,97){\vector(1,-1){14}} \put(74,92){\makebox(0,0){$1$}}
\put(73,97){\vector(1,-1){14}} \put(78,94){\makebox(0,0){$4$}}
\put(87,97){\vector(-1,-1){14}} \put(86,92){\makebox(0,0){$5$}}
\put(93,97){\vector(1,-1){14}} \put(94,92){\makebox(0,0){$1$}}
\put(107,97){\vector(-1,-1){14}} \put(102,94){\makebox(0,0){$2$}}
\put(109,97){\vector(-1,-1){14}} \put(106,92){\makebox(0,0){$5$}}
\put(70,80){\makebox(0,0){ABBCAC}}
\put(90,80){\makebox(0,0){ABCABC}}
\put(110,80){\makebox(0,0){ABACCB}}
\put(70,77){\vector(0,-1){14}} \put(67,70){\makebox(0,0){$4$}}
\put(90,77){\vector(0,-1){14}} \put(88,72){\makebox(0,0){$3$}}
\put(110,77){\vector(0,-1){14}} \put(113,70){\makebox(0,0){$2$}}
\put(73,77){\vector(1,-1){14}} \put(83,70){\makebox(0,0){$1$}}
\put(107,77){\vector(-1,-1){14}} \put(97,70){\makebox(0,0){$5$}}
\put(70,60){\makebox(0,0){ABBACC}}
\put(90,60){\makebox(0,0){ABACBC}}
\put(110,60){\makebox(0,0){AABCCB}}
\put(71,57){\vector(1,-2){7}} \put(72,50){\makebox(0,0){$1$}}
\put(73,57){\vector(1,-2){7}} \put(79,50){\makebox(0,0){$3$}}
\put(88,57){\vector(-1,-2){7}} \put(86,50){\makebox(0,0){$4$}}
\put(92,57){\vector(1,-2){7}} \put(94,50){\makebox(0,0){$2$}}
\put(107,57){\vector(-1,-2){7}} \put(101,50){\makebox(0,0){$3$}}
\put(109,57){\vector(-1,-2){7}} \put(109,50){\makebox(0,0){$5$}}
\put(80,40){\makebox(0,0){ABABCC}}
\put(100,40){\makebox(0,0){AABCBC}}
\put(82,37){\vector(1,-2){7}} \put(83,30){\makebox(0,0){$2$}}
\put(98,37){\vector(-1,-2){7}} \put(97,30){\makebox(0,0){$4$}}
\put(90,20){\makebox(0,0){AABBCC}}
\put(90,10){\makebox(0,0){Fig.3. ${\rm Sp}_6(\bbf)\backslash {\rm GL}_6(\bbf)/B$}}
\end{picture}

\begin{proposition} \ {\rm (i)} If $(i,j)\in S_0$ then $c_{i,j}=1$.

{\rm (ii)} We can write $S_0=\{(x_1,y_1),(x_2,y_2),\ldots,(x_s,y_s)\}$ with some $x_1<x_2<\cdots<x_s$ and $y_1<y_2<\cdots<y_s$ satisfying $\{x_1,\ldots,x_s\}\cap \{y_1,\ldots,y_s\}=\phi$.
\label{prop1.12}
\end{proposition}

Define $m=m(\mcf)=|\{(i,j)\mid c_{i,j}=1\mbox{ and }(i,j)\in S-S_0\}|$.

\begin{proposition} \ {\rm (i)} There exists a basis $v_1,\ldots,v_{2n}$ of $\bbf^{2n}$ satisfying {\rm (a)}, {\rm (b)} and {\rm (c)}$:$

\indent\indent {\rm (a)} $V_i=\bbf v_1\oplus\cdots\oplus \bbf v_i$ for $i=1,\ldots,2n$.

\indent\indent {\rm (b)} $\langle v_i,v_j\rangle=c_{i,j}$ for $i<j$.

\indent\indent {\rm (c)} $v_i\in W$ for $i\ne x_1,\ldots,x_s$ and $\langle v_{x_1},e_{2n}\rangle=\cdots =\langle v_{x_s},e_{2n}\rangle=1$.

{\rm (ii)} If $\bbf=\bbf_r$, then the number of bases satisfying the properties {\rm (a)}, {\rm (b)} and {\rm (c)} in {\rm (i)} is $(r-1)^{n-s}r^{n+\ell(\sigma)-m}$.
\label{prop1.13}
\end{proposition}

For a subset $I_{(A)}$ of $I$ with even number of elements, write
$$C(I_{(A)})=\{\mbox{symmetric permutation matrices }\{c_{i,j}\}_{i,j\in I_{(A)}}\mbox{ such that }c_{i,i}=0\mbox{ for }i\in I_{(A)}\}.$$

\begin{theorem} \ {\rm (i)} There exists a one-to-one correspondence between
$$\bigsqcup_{s=1}^n \bigsqcup_* C(I_{(A)})\mand Q_{2n}\backslash {\rm GL}_{2n}(\bbf)/B.$$
Here the disjoint union $*$ is taken for all the partitions $I=I_{(A)}\sqcup I_{(X)}\sqcup I_{(Y)}$ such that $|I_{(X)}|=|I_{(Y)}|=s$.

{\rm (ii)} $\displaystyle{|Q_{2n}\backslash {\rm GL}_{2n}(\bbf)/B|=\sum_{s=1}^n{(2n)!\over 2^{n-s}(s!)^2(n-s)!}}$.

{\rm (iii)} If $\bbf=\bbf_r$, then
\begin{align*}
|Q_{2n}\mcf| & ={|Q_{2n}|\over (r-1)^{n-s(\mcf)}r^{n+\ell(\sigma(\mcf))-m(\mcf)}} \\
& ={(r^2-1)(r^4-1)\cdots(r^{2n-2}-1)\over (r-1)^{n-s(\mcf)}} r^{n^2-n-\ell(\sigma(\mcf))+m(\mcf)}.
\end{align*}
\label{th1.14}
\end{theorem}


\begin{picture}(140,130)
\put(70,120){\makebox(0,0){YAAX}}
\put(67,117){\vector(-1,-1){14}} \put(56,110){\makebox(0,0){$1$}}
\put(73,117){\vector(1,-1){14}} \put(84,110){\makebox(0,0){$3$}}
\put(50,100){\makebox(0,0){AYAX}}
\put(70,100){\makebox(0,0){AYXA}}
\put(90,100){\makebox(0,0){YAXA}}
\put(48,97){\vector(-1,-2){7}} \put(41,90){\makebox(0,0){$2$}}
\put(52,97){\vector(1,-2){7}} \put(52,90){\makebox(0,0){$3$}}
\put(67,97){\vector(-1,-2){7}} \put(62,92){\makebox(0,0){$1$}}
\put(69,97){\vector(-1,-2){7}} \put(68,90){\makebox(0,0){$3$}}
\put(72,97){\vector(1,-2){7}} \put(80,87){\makebox(0,0){$2$}}
\put(87,97){\vector(-3,-2){21}} \put(80,95){\makebox(0,0){$1$}}
\put(93,97){\vector(1,-1){14}} \put(104,90){\makebox(0,0){$2$}}
\put(40,80){\makebox(0,0){AAYX}}
\put(60,80){\makebox(0,0){YYXX}}
\put(80,80){\makebox(0,0){AXYA}}
\put(110,80){\makebox(0,0){YXAA}}
\put(38,77){\vector(-1,-2){7}} \put(31,70){\makebox(0,0){$3$}}
\put(58,77){\vector(-1,-2){7}} \put(51,70){\makebox(0,0){$2$}}
\put(78,77){\vector(-1,-2){7}} \put(71,70){\makebox(0,0){$3$}}
\put(82,77){\vector(1,-2){7}} \put(89,70){\makebox(0,0){$1$}}
\put(113,77){\vector(1,-1){14}} \put(124,70){\makebox(0,0){$1$}}
\put(30,60){\makebox(0,0){AAXY}}
\put(50,60){\makebox(0,0){YXYX}}
\put(70,60){\makebox(0,0){AXAY}}
\put(90,60){\makebox(0,0){XAYA}}
\put(110,60){\makebox(0,0){XAAY}}
\put(130,60){\makebox(0,0){XYAA}}
\put(33,57){\vector(1,-1){14}} \put(36,50){\makebox(0,0){$2$}}
\put(50,57){\vector(0,-1){14}} \put(47,50){\makebox(0,0){$3$}}
\put(67,57){\vector(-1,-1){14}} \put(64,50){\makebox(0,0){$2$}}
\put(56,58){\vector(3,-1){48}} \put(73,49){\makebox(0,0){$1$}}
\put(73,57){\vector(1,-1){14}} \put(78,55){\makebox(0,0){$1$}}
\put(90,57){\vector(0,-1){14}} \put(88,52){\makebox(0,0){$3$}}
\put(106,57){\vector(-1,-1){14}} \put(102,56){\makebox(0,0){$1$}}
\put(108,57){\vector(-1,-1){14}} \put(106,52){\makebox(0,0){$3$}}
\put(93,57){\vector(1,-1){14}} \put(95,52){\makebox(0,0){$2$}}
\put(127,57){\vector(-1,-1){14}} \put(124,50){\makebox(0,0){$2$}}
\put(50,40){\makebox(0,0){YXXY}}
\put(90,40){\makebox(0,0){XXYY}}
\put(110,40){\makebox(0,0){XYYX}}
\put(56,37){\vector(2,-1){28}} \put(67,29){\makebox(0,0){$1$}}
\put(90,37){\vector(0,-1){14}} \put(87,30){\makebox(0,0){$2$}}
\put(107,37){\vector(-1,-1){14}} \put(104,30){\makebox(0,0){$3$}}
\put(90,20){\makebox(0,0){XYXY}}
\put(80,10){\makebox(0,0){Fig.4. $Q_4\backslash {\rm GL}_4(\bbf)/B$}}
\end{picture}

For $n=1,2,3,4$, the number of orbits $|Q_{2n}\backslash {\rm GL}_{2n}(\bbf)/B|$ is as follows.

\bigskip
\centerline{
\vbox{\offinterlineskip
\hrule
\halign{&\vrule#&\strut\quad$\hfil#\hfil$\quad\cr
height2pt&\omit&&\omit&&\omit&&\omit&&\omit&\cr
& n && 1 && 2 && 3 && 4 &\cr
height2pt&\omit&&\omit&&\omit&&\omit&&\omit&\cr
\noalign{\hrule}
height4pt&\omit&&\omit&&\omit&&\omit&&\omit&\cr
& |Q_{2n}\backslash {\rm GL}_{2n}(\bbf)/B| && 2 && 18 && 200 && 2730 &\cr
height4pt&\omit&&\omit&&\omit&&\omit&&\omit&\cr}
\hrule}
}

We can express orbits by ``ABXY-symbols''. When $n=2$, the orbit structure is as in Fig.4. In the diagram, the symbol XYXY implies the $Q_4$-orbit containing the flag $\mcf$ such that $S_0(\mcf)=\{(1,2),(3,4)\}$, for example.

Finally we will solve Problem (D) only restricting Problem (C) to a subgroup of ${\rm GL}_{2n}(\bbf)$. Retain the notations in Problem (C). Define a subspace $W'=\bbf e_1\oplus\cdots\oplus \bbf e_{2n-1}=(\bbf e_1)^\perp$ of $\bbf^{2n}$. Then $Q'_{2n}=\{g\in Q_{2n}\mid gW'=W'\}$ is written as
$$Q'_{2n}=\{g\in H\mid ge_1=e_1\mbox{ and }ge_{2n}=e_{2n}\}\cong 1\times{\rm Sp}_{2n-2}(\bbf)\subset {\rm GL}(W').$$
Consider the variety $M'$ consisting of full flags $V_1\subset\cdots\subset V_{2n-1}$ satisfying $V_{2n-1}=W'$. Then $M'$ is the full flag variety of ${\rm GL}(W')$. Two flags in $M'$ are in the same $Q'_{2n}$-orbit if and only if they are in the same $Q_{2n}$-orbit.

For a full flag $V_1\subset\cdots\subset V_{2n-1}$ in $M'$, let $i$ be the least integer such that $V_i\supset \bbf e_1$. Then the pair $(i,2n)$ is contained in $S_0$ since $V_i\cap V_{2n-1}^\perp=\bbf e_1\not\subset W$. This implies  $x_s=i$ and $y_s=2n$. Thus we have:

\begin{theorem} \ {\rm (i)} There exists a one-to-one correspondence between
$$\bigsqcup_{s=1}^n \bigsqcup_* C(I_{(A)})\mand 1\times {\rm Sp}_{2n-2}(\bbf)\backslash {\rm GL}_{2n-1}(\bbf)/B.$$
Here the disjoint union $*$ is taken for all the partitions $I=\{1,\ldots,2n-1\}=I_{(A)}\sqcup I_{(X)}\sqcup I_{(Y)}$ such that $|I_{(X)}|=s$ and that $|I_{(Y)}|=s-1$.

{\rm (ii)} $\displaystyle{|1\times {\rm Sp}_{2n-2}(\bbf)\backslash {\rm GL}_{2n-1}(\bbf)/B|=\sum_{s=1}^n{(2n-1)!\over 2^{n-s}s!(s-1)!(n-s)!}}.$

{\rm (iii)} If $\bbf=\bbf_r$, then
\begin{align*}
|(1\times {\rm Sp}_{2n-2}(\bbf))\mcf| & ={|{\rm Sp}_{2n-2}(\bbf)|\over (r-1)^{n-s(\mcf)}r^{n+\ell(\sigma(\mcf))-m(\mcf)}} \\
& ={(r^2-1)(r^4-1)\cdots(r^{2n-2}-1)\over (r-1)^{n-s(\mcf)}} r^{(n-1)^2-n-\ell(\sigma(\mcf))+m(\mcf)}
\end{align*}
Here the invariants $\sigma(\mcf)$ and $m(\mcf)$ are defined for the natural extension of the full flag $\mcf$ to $\bbf^{2n}$ as is explained above.
\label{th1.15}
\end{theorem}

For $n\le 5$, the number of orbits $|1\times {\rm Sp}_{2n-2}(\bbf)\backslash {\rm GL}_{2n-1}(\bbf)/B|$ is as follows.

\bigskip
\centerline{
\vbox{\offinterlineskip
\hrule
\halign{&\vrule#&\strut\quad$\hfil#\hfil$\quad\cr
height2pt&\omit&&\omit&&\omit&&\omit&&\omit&&\omit&\cr
& n && 1 && 2 && 3 && 4 && 5 &\cr
height2pt&\omit&&\omit&&\omit&&\omit&&\omit&&\omit&\cr
\noalign{\hrule}
height4pt&\omit&&\omit&&\omit&&\omit&&\omit&&\omit&\cr
& |1\times {\rm Sp}_{2n-2}(\mathbb{F})\backslash {\rm GL}_{2n-1}(\bbf)/B| && 1 && 6 && 55 && 665 && 9891 &\cr
height4pt&\omit&&\omit&&\omit&&\omit&&\omit&&\omit&\cr}
\hrule}
}

We can express $1\times {\rm Sp}_2(\bbf)$-orbits on ${\rm GL}_3(\bbf)/B$ as follows:

\begin{picture}(100,65)(0,-10)
\put(50,45){\makebox(0,0){AAX}}
\put(70,45){\makebox(0,0){AXA}}
\put(90,45){\makebox(0,0){XAA}}
\put(60,25){\makebox(0,0){YXX}}
\put(80,25){\makebox(0,0){XXY}}
\put(70,5){\makebox(0,0){XYX}}
\put(51.5,42){\vector(1,-2){7}}
\put(68.5,42){\vector(-1,-2){7}}
\put(71.5,42){\vector(1,-2){7}}
\put(88.5,42){\vector(-1,-2){7}}
\put(61.5,22){\vector(1,-2){7}}
\put(78.5,22){\vector(-1,-2){7}}
\put(52,35){\makebox(0,0){2}}
\put(68,35){\makebox(0,0){2}}
\put(72,35){\makebox(0,0){1}}
\put(88,35){\makebox(0,0){1}}
\put(62,15){\makebox(0,0){1}}
\put(78,15){\makebox(0,0){2}}
\put(70,-5){\makebox(0,0){Fig.5. $1\times {\rm Sp}_2(\bbf)\backslash {\rm GL}_3(\bbf)/B$}}
\end{picture}

\noindent We have only to extract from the diagram of $Q_4\backslash {\rm GL}_4(\bbf)/B$ six symbols containing the letter ``Y'' as the fourth letter and then delete these ``Y''. Clearly the orbit structure is the same as ${\rm GL}_2(\bbf)\times {\rm GL}_1(\bbf)\backslash {\rm GL}_3(\bbf)/B$ (c.f. \cite{MO}, Fig.5).

\subsection{Expression by symbols and number of orbits}

By Theorem \ref{th1.6}, Corollary \ref{cor1.11}, Theorem \ref{th1.14}, Theorem \ref{th1.15} and Proposition \ref{prop4.3}, we can attach each $G$-orbit $Gt$ ($t=(U_0,U_d,\mcf),\ d=n-a-b$) on $\mct_0=M\times M\times M_0$ a ``word'' $w=\ell_1\cdots\ell_n$ consisting of letters $\ell_i$ as follows. Write $I=\{1,\ldots,n\}=I_{(\alpha)}\sqcup I_{(\beta)}\sqcup I_{(\gamma)}\sqcup I_{(\delta)}$ as in Section 1.2. Then
\begin{align*}
i\in I_{(\alpha)} & \Longrightarrow \ell_i=\alpha, \quad
i\in I_{(\beta)} \Longrightarrow \ell_i=\beta, \quad
i\in I_{(\gamma)} \Longrightarrow \ell_i=+,-\mbox{ or a, b},\ldots, \\
i\in I_{(\delta)} & \Longrightarrow \ell_i=\mbox{X,Y or A,B},\ldots.
\end{align*}
Here the subword $w_{(\gamma)}=\ell_{\gamma_1}\cdots\ell_{\gamma_c}$ expresses an $L_+\times L_-$-orbit of the full flag
$V_{\gamma_1}\cap W^0\subset\cdots\subset V_{\gamma_c}\cap W^0$
in $W^0=U_{(+)}\oplus U_{(-)}$ as in Section 1.3 and Section 4. On the other hand, the subword $w_{(\delta)}=\ell_{\delta_1}\cdots\ell_{\delta_{c_0}}$ expresses an $L_V\cap L_0$-orbit of the full flag
$V_{\delta_1}\cap U_{(0)}\subset\cdots\subset V_{\delta_{c_0}}\cap U_{(0)}$
in $U_{(0)}$ as in Section 1.3.

\begin{example} (The case of $n=2$) We can describe $\displaystyle{G\backslash \mct_0 \cong \bigsqcup_{d=0}^2 R_d\backslash M_0}$ as in Fig.6 when $n=2$.

\bigskip
\noindent Notation: Let $p_i: M_0\to M_i\ (i=1,2)$ be the canonical projections where $M_1=\{V_2\mid \dim V_2=2\}\ (=M)$ and $M_2=\{V_1\mid \dim V_1=1\}$ are partial flag varieties. Then for two $R_d$-orbits $S_1$ and $S_2$ on $M_0$, we write $S_1\xrightarrow{i} S_2$ when $p_i(S_1)=p_i(S_2)$ and $\dim S_1+1=\dim S_2$.
\end{example}

\begin{remark} \ Suppose $\bbf=\bc$. Then $G={\rm SO}_5(\bc)\cong {\rm Sp}_4(\bc)/\mathbb{Z}_2$ in this case. So the orbit structure is the same as the symplectic triple flag variety of the shape $(121)(121)(1^4)$ given in \cite{MWZ2}.
\end{remark}

Using these symbols, we can easily count the number of orbits as follows. Let $\xi(k)$ denote the number of words consisting of $k$ letters which do not contain the four letters $\alpha,\beta,+$ and $-$.

\begin{picture}(140,115)
\put(10,110){\makebox(0,0){$\alpha\alpha$}}
\put(10,107){\vector(0,-1){14}} \put(13,100){\makebox(0,0){$2$}}
\put(10,90){\makebox(0,0){$\alpha\beta$}}
\put(10,87){\vector(0,-1){14}} \put(13,80){\makebox(0,0){$1$}}
\put(10,70){\makebox(0,0){$\beta\alpha$}}
\put(10,67){\vector(0,-1){14}} \put(13,60){\makebox(0,0){$2$}}
\put(10,50){\makebox(0,0){$\beta\beta$}}
\put(10,20){\makebox(0,0){$d=0$}}
\put(40,110){\makebox(0,0){$\alpha+$}}
\put(60,110){\makebox(0,0){$\alpha-$}}
\put(38,107){\vector(-1,-2){7}} \put(32,100){\makebox(0,0){$1$}}
\put(42,107){\vector(1,-2){7}} \put(43,100){\makebox(0,0){$2$}}
\put(58,107){\vector(-1,-2){7}} \put(57,100){\makebox(0,0){$2$}}
\put(62,107){\vector(1,-2){7}} \put(68,100){\makebox(0,0){$1$}}
\put(30,90){\makebox(0,0){$+\alpha$}}
\put(50,90){\makebox(0,0){$\alpha{\rm X}$}}
\put(70,90){\makebox(0,0){$-\alpha$}}
\put(30,87){\vector(0,-1){14}} \put(27,80){\makebox(0,0){$2$}}
\put(50,87){\vector(0,-1){14}} \put(53,80){\makebox(0,0){$1$}}
\put(70,87){\vector(0,-1){14}} \put(73,80){\makebox(0,0){$2$}}
\put(30,70){\makebox(0,0){$+\beta$}}
\put(50,70){\makebox(0,0){${\rm X}\alpha$}}
\put(70,70){\makebox(0,0){$-\beta$}}
\put(30,67){\vector(0,-1){14}} \put(27,60){\makebox(0,0){$1$}}
\put(50,67){\vector(0,-1){14}} \put(53,60){\makebox(0,0){$2$}}
\put(70,67){\vector(0,-1){14}} \put(73,60){\makebox(0,0){$1$}}
\put(30,50){\makebox(0,0){$\beta+$}}
\put(50,50){\makebox(0,0){${\rm X}\beta$}}
\put(70,50){\makebox(0,0){$\beta-$}}
\put(33,47){\vector(1,-1){14}} \put(37,40){\makebox(0,0){$2$}}
\put(50,47){\vector(0,-1){14}} \put(53,40){\makebox(0,0){$1$}}
\put(67,47){\vector(-1,-1){14}} \put(63,40){\makebox(0,0){$2$}}
\put(50,30){\makebox(0,0){$\beta{\rm X}$}}
\put(50,20){\makebox(0,0){$d=1$}}
\put(80,110){\makebox(0,0){$++$}}
\put(100,110){\makebox(0,0){$+-$}}
\put(120,110){\makebox(0,0){$-+$}}
\put(140,110){\makebox(0,0){$--$}}
\put(82,107){\vector(1,-2){7}} \put(83,100){\makebox(0,0){$2$}}
\put(98,107){\vector(-1,-2){7}} \put(97,100){\makebox(0,0){$2$}}
\put(102,107){\vector(1,-2){7}} \put(103,100){\makebox(0,0){$1$}}
\put(118,107){\vector(-1,-2){7}} \put(117,100){\makebox(0,0){$1$}}
\put(122,107){\vector(1,-2){7}} \put(123,100){\makebox(0,0){$2$}}
\put(138,107){\vector(-1,-2){7}} \put(137,100){\makebox(0,0){$2$}}
\put(90,90){\makebox(0,0){$+{\rm X}$}}
\put(110,90){\makebox(0,0){${\rm aa}$}}
\put(120,90){\makebox(0,0){${\rm AA}$}}
\put(130,90){\makebox(0,0){$-{\rm X}$}}
\put(90,87){\vector(0,-1){14}} \put(87,80){\makebox(0,0){$1$}}
\put(110,87){\vector(0,-1){14}} \put(107,80){\makebox(0,0){$2$}}
\put(119,87){\vector(-1,-2){7}} \put(118,80){\makebox(0,0){$2$}}
\put(130,87){\vector(0,-1){14}} \put(133,80){\makebox(0,0){$1$}}
\put(90,70){\makebox(0,0){${\rm X}+$}}
\put(110,70){\makebox(0,0){YX}}
\put(130,70){\makebox(0,0){${\rm X}-$}}
\put(93,67){\vector(1,-1){14}} \put(97,60){\makebox(0,0){$2$}}
\put(110,67){\vector(0,-1){14}} \put(113,60){\makebox(0,0){$1$}}
\put(127,67){\vector(-1,-1){14}} \put(123,60){\makebox(0,0){$2$}}
\put(110,50){\makebox(0,0){XY}}
\put(110,20){\makebox(0,0){$d=2$}}
\put(70,10){\makebox(0,0){Fig.6. $\displaystyle{\bigsqcup_{d=0}^2 R_d\backslash {\rm SO}_5(\bbf)/B}$}}
\end{picture}

\begin{lemma} \ $\displaystyle{\xi(2k)=\sum_{s=0}^k {(2k)!\over (s!)^2(k-s)!}}$ and $\displaystyle{\xi(2k-1)=\sum_{s=1}^k {(2k-1)!\over s!(s-1)!(k-s)!}}$.
\label{lem1.18}
\end{lemma}

\begin{theorem} \ {\rm (i)} $\displaystyle{|G\backslash \mct_0|= \sum_{k=0}^n4^{n-k}{n\choose k}\xi(k)}$.

{\rm (ii)} $\displaystyle{|{\rm GL}_n(\bbf)\backslash M_0|= \sum_{k=0}^n2^{n-k}{n\choose k}\xi(k)}$.
\label{th1.19}
\end{theorem}

For $n=1,2,3,4$, the number of orbits is as follows.

\bigskip
\centerline{
\vbox{\offinterlineskip
\hrule
\halign{&\vrule#&\strut\quad$\hfil#\hfil$\quad\cr
height2pt&\omit&&\omit&&\omit&&\omit&&\omit&\cr
& n && 1 && 2 && 3 && 4 &\cr
height2pt&\omit&&\omit&&\omit&&\omit&&\omit&\cr
\noalign{\hrule}
height4pt&\omit&&\omit&&\omit&&\omit&&\omit&\cr
& |G\backslash \mct_0| && 5 && 28 && 169 && 1082 &\cr
height4pt&\omit&&\omit&&\omit&&\omit&&\omit&\cr
& |{\rm GL}_n(\bbf)\backslash M_0| && 3 && 12 && 53 && 258 &\cr
height4pt&\omit&&\omit&&\omit&&\omit&&\omit&\cr}
\hrule}
}

\subsection{The case of ${\rm SO}_{2n}(\bbf)$}

Let $G'$ and $\widetilde{G}'$ be the subgroups of $G$ defined by
$$G'=\{g\in G\mid g e_{n+1}=e_{n+1}\}\mand \widetilde{G}'=\{g\in G\mid g e_{n+1}=\pm e_{n+1}\},$$
respectively. Then they are isomorphic to the split special orthogonal group and the split orthogonal group of degree $2n$, respectively.
Let $M'$ be the subvariety of $M$ defined by $M'=\{V\in M\mid (V,e_{n+1})=\{0\}\}$.

Then $M'$ is a homogeneous space of $\widetilde{G}'$ consisting of two $G'$-orbits $M^0=G'U_0$ and $M^1=G'U_1\ (U_1=\bbf e_1\oplus\cdots\oplus \bbf e_{n-1}\oplus \bbf e_{n+2})$. Note that $V$ and $V'$ in $M'$ are contained in the same $G'$-orbit if and only if $n-\dim(V\cap V')$ is even. Define a subvariety $M'_0=\{\mcf:V_1\subset\cdots\subset V_n \mid V_n\in M'\}$ of $M_0$. Then $M'_0$ is also a homogeneous space of $\widetilde{G}'$ consisting of two $G'$-orbits $M^0_0$ and $M^1_0$.

\begin{theorem} \ Let $t=(V_{(1)},V_{(2)},V_{(3)})$ be an element of $\mct'=M'\times M'\times M'$. Define
\begin{align*}
a & =a(t)=\dim (V_{(1)}\cap V_{(2)}\cap V_{(3)}),\ b=b(t)=\dim (V_{(1)}\cap V_{(2)}) -a, \\
c_+ & =c_+(t)=\dim (V_{(1)}\cap V_{(3)}) -a,\ c_-=c_-(t)=\dim (V_{(2)}\cap V_{(3)}) -a, \\
c_0 & =c_0(t)=n-a-b-c_+-c_- \\
\mand \varepsilon & =\varepsilon(t)=\dim(V_{(1)}+V_{(2)}+V_{(3)})+\dim(V_{(1)}\cap V_{(2)}\cap V_{(3)})-2n\in \{0,1\}.
\end{align*}
Then we have$:$

{\rm (i)} $\varepsilon=0,\ c_0$ is even and $t\in \widetilde{G}'(U_0,U_{n-a-b},V(a,b,c_+,c_-)_{\rm even}^0)$.

{\rm (ii)} If $\bbf=\bbf_r$, then $\displaystyle{|\widetilde{G}'t|=2|G't|=|M'|{r^{(n-a)(n-a-1)}[r]_n\psi_{c_0}^0(r)\over [r]_a[r]_b[r]_{c_+}[r]_{c_0}[r]_{c_-}}}$.
\label{th1.20}
\end{theorem}

\begin{corollary} \ {\rm (i)} $\displaystyle{|\widetilde{G}'\backslash \mct'| ={|G'\backslash \mct'|\over 2}=\sum_{k=0}^{[n/2]} {n-2k+3\choose 3}}$.

{\rm (ii)} $\displaystyle{|G'\backslash M^{\nu_1}\times M^{\nu_2}\times M^{\nu_3}| =\sum_{k=0}^{[n/2]} {k+3\choose 3} +\sum_{k=0}^{[(n-3)/2]} {k+3\choose 3}}$ if $\nu_1=\nu_2=\nu_3$.

{\rm (iii)} $\displaystyle{|G'\backslash M^{\nu_1}\times M^{\nu_2}\times M^{\nu_3}| =
\sum_{k=0}^{[(n-1)/2]} {k+3\choose 3} +\sum_{k=0}^{[(n-2)/2]} {k+3\choose 3}}$ if $\nu_i\ne \nu_j$ for some $i,j=1,2,3$.
\end{corollary}

For $n=2,3,4,5$, the number of orbits is as follows.

\bigskip
\centerline{
\vbox{\offinterlineskip
\hrule
\halign{&\vrule#&\strut\quad$\hfil#\hfil$\quad\cr
height2pt&\omit&&\omit&&\omit&&\omit&&\omit&\cr
& n && 2 && 3 && 4 && 5 &\cr
height2pt&\omit&&\omit&&\omit&&\omit&&\omit&\cr
\noalign{\hrule}
height4pt&\omit&&\omit&&\omit&&\omit&&\omit&\cr
& |\widetilde{G}'\backslash \mct'|=|G'\backslash \mct'|/2 && 11 && 24 && 46 && 80 &\cr
height4pt&\omit&&\omit&&\omit&&\omit&&\omit&\cr
& |G'\backslash M^{\nu_1}\times M^{\nu_2}\times M^{\nu_3}|\ (\nu_1=\nu_2=\nu_3) && 5 && 6 && 16 && 20 &\cr
height4pt&\omit&&\omit&&\omit&&\omit&&\omit&\cr
& |G'\backslash M^{\nu_1}\times M^{\nu_2}\times M^{\nu_3}|\ (\nu_i\ne \nu_j\mbox{ for some }i,j) && 2 && 6 && 10 && 20 &\cr
height4pt&\omit&&\omit&&\omit&&\omit&&\omit&\cr}
\hrule}
}

Let $t=(U_0,U_d,V)$ with $d=n-a-b$ and $V=V(a,b,c_+,c_-)_{\rm even}^0$. The following proposition is proved in the same way as Theorem \ref{th1.6}.

\begin{proposition} \ {\rm (i)} For every full flag $\mcf$ in $M_0(V)\cap M'_0$, there exists a $g\in R(t)\cap G'$ such that $g\mcf$ is standard.

{\rm (ii)} Let $\mcf$ and $\mcf'$ be two standard full flags in $M_0(V)\cap M'_0$. Then $\mcf'=g\mcf$ for some $g\in R(t)\cap G'\Longrightarrow \mcf'=g_L\mcf$ for some $g_L\in L_V\cap G'$.

{\rm (iii)} Let $\mcf$ be a standard full flag in $M_0(V)\cap M'_0$. If $\bbf=\bbf_r$, then $|(R(t)\cap G')\mcf|=[r]_a[r]_br^{\ell(\tau(\mcf))}|(L_V\cap G')\mcf|$.
\end{proposition}

As in Section 1.4, we can express $G'$-orbits on $\mct'_0=M'\times M'\times M'_0$ by words with letters $\alpha,\ \beta,\ +,\ -,\ {\rm a,b},\ldots$ and ${\rm A,B},\ldots$. The following corollary is proved in the same way as Theorem \ref{th1.19}.

\begin{corollary} \ {\rm (i)} $\displaystyle{|\widetilde{G}'\backslash \mct'_0|={|G'\backslash \mct'_0|\over 2}=\sum_{k=0}^{[n/2]} 4^{n-2k}{n\choose 2k}{(2k)!\over k!}=\sum_{k=0}^{[n/2]} {4^{n-2k}n!\over k!(n-2k)!}}$.

{\rm (ii)} $\displaystyle{|G'\backslash M^{\nu_1}\times M^{\nu_2}\times M^{\nu_3}_0|={|\widetilde{G}'\backslash \mct'_0|\over 4}+\mu{n!\over 4(n/2)!}}$ where
$$\mu=\begin{cases} 0 & \text{if $n$ is odd}, \\ 1 & \text{if $n$ is even and $\nu_1=\nu_2$}, \\ -1 & \text{if $n$ is even and $\nu_1\ne\nu_2$}. \end{cases}$$

{\rm (iii)} $\displaystyle{|{\rm GL}_n(\bbf)\backslash M'_0|=2|{\rm GL}_n(\bbf)\backslash M^0_0|=\sum_{k=0}^{[n/2]} {2^{n-2k}n!\over k!(n-2k)!}}$.
\end{corollary}

For $n=2,3,4,5$, the number of orbits is as follows.

\bigskip
\centerline{
\vbox{\offinterlineskip
\hrule
\halign{&\vrule#&\strut\quad$\hfil#\hfil$\quad\cr
height2pt&\omit&&\omit&&\omit&&\omit&&\omit&\cr
& n && 2 && 3 && 4 && 5 &\cr
height2pt&\omit&&\omit&&\omit&&\omit&&\omit&\cr
\noalign{\hrule}
height4pt&\omit&&\omit&&\omit&&\omit&&\omit&\cr
& |\widetilde{G}'\backslash \mct'_0|=|G'\backslash \mct'_0|/2 && 18 && 88 && 460 && 2544 &\cr
height4pt&\omit&&\omit&&\omit&&\omit&&\omit&\cr
& |G'\backslash M^{\nu_1}\times M^{\nu_2}\times M^{\nu_3}_0|\ (\nu_1=\nu_2) && 5 && 22 && 118 && 636 &\cr
height4pt&\omit&&\omit&&\omit&&\omit&&\omit&\cr
& |G'\backslash M^{\nu_1}\times M^{\nu_2}\times M^{\nu_3}_0|\ (\nu_1\ne \nu_2) && 4 && 22 && 112 && 636 &\cr
height4pt&\omit&&\omit&&\omit&&\omit&&\omit&\cr
& |{\rm GL}_n(\bbf)\backslash M'_0|=2|{\rm GL}_n(\bbf)\backslash M^0_0| && 6 && 20 && 76 && 312 &\cr
height4pt&\omit&&\omit&&\omit&&\omit&&\omit&\cr}
\hrule}
}

\begin{remark} \ When $\bbf=\bc$, ${\rm GL}_n(\bc)$ is a symmetric subgroup of ${\rm SO}_{2n}(\bc)$. The structure of ${\rm GL}_n(\bc)\backslash M^0_0\cong {\rm GL}_n(\bc)\backslash {\rm SO}_{2n}(\bc)/B$ is described in \cite{MO} (Fig.19 and Fig.20) for $n=3$ and $4$. (In \cite{MO}, read $GL(n,\bc)$ for $\bc^\times \times PSL(n,\bc)$. For $A++$ in Fig.20 read $AA+$.)
\end{remark}

\section{Orbits on $M$ and $M_0$}

\subsection{Preliminaries}

First we prepare some results on orbits on $M$ and $M_0$ which follow essentially from the Bruhat decompositions for the Chevalley-type groups. Since we need more explicit results, we will prove them by elementary arguments.

Write $W_0=\bbf e_1\oplus\cdots\oplus \bbf e_{n-d},\ W_1=\bbf e_{n-d+1}\oplus\cdots\oplus \bbf e_{n+d+1}$ and $W_2=\bbf e_{n+d+2}\oplus\cdots\oplus \bbf e_{2n+1}$.
Then $W_0^\perp=W_0\oplus W_1$. Define a maximal parabolic subgroup
$P_{W_0}=\{g\in G\mid gW_0=W_0\}$
of $G$. Let $N_{W_0},\ L_{W_0}$ and $L_{W_1}$ be subgroups of $P_{W_0}$ defined by
\begin{align*}
N_{W_0} & =\left\{\bp I_m & * & * \\ 0 & I_{2d+1} & * \\ 0 & 0 & I_m \ep\right\}, \ 
L_{W_0} =\left\{\bp A & 0 & 0 \\ 0 & I_{2d+1} & 0 \\ 0 & 0 & J_m\,{}^tA^{-1}J_m  \ep \Bigm| A\in{\rm GL}_m(\bbf)\right\} \\
& \mand L_{W_1} =\{g\in G\mid gv=v\mbox{ for all }v\in W_0\oplus W_2\},
\end{align*}
respectively, where $m=n-d$. Then $N_{W_0}$ is the unipotent radical of $P_{W_0}$ and $L_{W_0}L_{W_1}\cong L_{W_0}\times L_{W_1}$ is a Levi subgroup of $P_{W_0}$. The subgroup $L_{W_1}$ is identified with the special orthogonal group for $W_1$. Define
$$Q=L_{W_0}N_{W_0}=N_{W_0}L_{W_0}.$$
Write
$$g(X,Z)=\bp I_m & -J_m\,{}^tXJ_{2d+1} & Z \\ 0 & I_{2d+1} & X \\ 0 & 0 & I_m \ep$$
where $X=\{x_{i,j}\}$ is a $(2d+1)\times m$ matrix and $Z=\{z_{i,j}\}$ is an $m\times m$ matrix. Then we can write
\begin{equation}
N_{W_0} =\{g(X,Z)\mid Z+J_m\,{}^tZJ_m=-J_m\,{}^tXJ_{2d+1}X\}. \label{eq2.0}
\end{equation}
Noting that
\begin{equation}
z_{m+1-j,j}=-{1\over 2}\sum_{k=1}^{2d+1} x_{k,j}x_{2d+2-k,j} \label{eq2.2'}
\end{equation}
for $j=1,\ldots,m$ and
\begin{equation}
z_{i,j}=-z_{m+1-j,m+1-i}-\sum_{k=1}^{2d+1} x_{k,j}x_{2d+2-k,m+1-i} \label{eq2.3}
\end{equation}
for $i+j\ge m+2$, we have:

\begin{lemma} \ There exists a bijection between $\bbf^{m(2d+1)+m(m-1)/2}$ and $N_{W_0}$ given by
$$(X,\{z_{i,j}\}_{i+j\le m})\mapsto g(X,Z).$$
\label{lem2.1}
\end{lemma}

\begin{definition} \ A maximally isotropic subspace $V$ in $\bbf^{2n+1}$ is called standard if
\begin{align*}
V & =(V\cap W_0)\oplus (V\cap W_1)\oplus (V\cap W_2), \\
V\cap W_0 & =U_{(\alpha)}\ (=\bbf e_1\oplus\cdots\oplus \bbf e_a) \\
\mbox{and}\quad V\cap W_2 & =U_{(\beta)}\ (=\bbf e_{n+d+2}\oplus\cdots\oplus \bbf e_{2n-a+1})
\end{align*}
with some $a=0,\ldots,m=n-d$.
\end{definition}

\begin{proposition} \ {\rm (i)} For every $V\in M$, there exists a $g\in Q$ such that $gV$ is standard.

{\rm (ii)} Let $V$ and $V'$ be two standard elements in $M$. If $V'=gV$ for some $g=g_Qg_L\in P_{W_0}=QL_{W_1}$, then $V'=g_L V$.

{\rm (iii)} When $\bbf=\bbf_r$, $\displaystyle{|QV|=r^{((n-a)(n-a+1)-d(d+1))/2}{[r]_{a+b}\over [r]_a[r]_b}}$ for $V\in M$. Here $a=\dim(V\cap W_0),\ b=n-d-a$ and
$[r]_k=(r+1)(r^2+r+1)\cdots(r^{k-1}+r^{k-2}+\cdots+1)$.
\label{prop2.3}
\end{proposition}

\begin{proof} (i) Let $\pi_2:\bbf^{2n+1}\to W_2$ denote the projection with respect to the direct sum decomposition $\bbf^{2n+1}=W_0\oplus W_1\oplus W_2$. By the action of $L_{W_0}\cong {\rm GL}(W_0)$, we may assume $V\cap W_0=U_{(\alpha)}$. It is equivalent to $\pi_2(V)=U_{(\beta)}$. Take vectors $u_1,\ldots,u_b\in V$ such that $\pi_2(u_j)=e_{n+d+1+j}$ for $j=1,\ldots,b$. Then we can write
$$u_j=e_{n+d+1+j}+\sum_{i=1}^m \widetilde{z}_{i,j}e_i +\sum_{i=1}^{2d+1} \widetilde{x}_{i,j}e_{m+i}$$
with some $\widetilde{z}_{i,j},\widetilde{x}_{i,j}\in\bbf$. It follows from the condition $(u_j,u_k)=0$ for $j,k=1,\ldots,b$ that
$$\widetilde{z}_{m+1-j,j}=-{1\over 2}\sum_{k=1}^{2d+1} \widetilde{x}_{k,j}\widetilde{x}_{2d+2-k,j}$$
for $j=1,\ldots,b$ and that
$$\widetilde{z}_{i,j}=-\widetilde{z}_{m+1-j,m+1-i}-\sum_{k=1}^{2d+1} \widetilde{x}_{k,j}\widetilde{x}_{2d+2-k,m+1-i}$$
for $(i,j)\in\{(i,j)\mid i+j\ge m+2,\ j\le b\}$. Hence we can take $X=\{x_{i,j}\}$ and $Z=\{z_{i,j}\}$ satisfying (\ref{eq2.2'}) and (\ref{eq2.3}) so that
$x_{i,j}=\widetilde{x}_{i,j}$ and that $z_{i,j}=\widetilde{z}_{i,j}$
for $j=1,\ldots,b$. Take $g=g(X,Z)\in N_{W_0}$. Then we have
$$g^{-1}V=(g^{-1}V\cap (W_0\oplus W_1))\oplus(g^{-1}V\cap W_2)\mand g^{-1}V\cap W_2=U_{(\beta)}.$$
Since $g^{-1}V\cap (W_0\oplus W_1)\perp g^{-1}V\cap W_2$, we have
$$g^{-1}V\cap (W_0\oplus W_1)\subset U_{(\alpha)}\oplus W_1.$$
Hence
$$g^{-1}V\cap (W_0\oplus W_1)=(g^{-1}V\cap W_0)\oplus (g^{-1}V\cap W_1)$$
with $g^{-1}V\cap W_0=V\cap W_0=U_{(\alpha)}$.

(ii) The condition $V'=gV$ implies $\dim(V\cap W_0)=\dim(V'\cap W_0)$. So we have
$$V_i\cap W_0=V'_i\cap W_0=U_{(\alpha)}\mand V_i\cap W_2=V'_i\cap W_2=U_{(\beta)}.$$
We have only to show that $g_L(V\cap W_1)=V'\cap W_1$. Let $v$ be an element of $V\cap W_1$. Then we have
$$gv\in V'\cap W_0^\perp=(V'\cap W_0)\oplus (V'\cap W_1).$$
So we can write $gv=v_0+v_1$ with some $v_0\in V'\cap W_0$ and $v_1\in V'\cap W_1$. Since $g_L v\in W_1$ and since $gv=g_Qg_L v\in g_L v+W_0$, it follows that
$g_L v=v_1$.
Thus we have proved $g_L(V\cap W_1)=V'\cap W_1$ since $\dim(V\cap W_1)=\dim(V'\cap W_1)$.

(iii) By (i), we may assume that $V$ is standard. Note that we may consider $\ell V$ instead of $V$ with some $\ell\in L_{W_1}$ because $\ell Q\ell^{-1}=Q$. So we may assume
$$V\cap W_1=\bbf e_{m+1}\oplus\cdots\oplus \bbf e_n.$$

Let $g=g(X,Z)$ be an element of $N_{W_0}$ such that $gV=V$. Since $g(V\cap W_2)\subset V$, we have
\begin{equation}
i\ge d+1,\ j\le b\Longrightarrow x_{i,j}=0 \label{eq2.4'}
\end{equation}
and
\begin{equation}
i\ge a+1,\ j\le b\Longrightarrow z_{i,j}=0. \label{eq2.5}
\end{equation}
It follows from (\ref{eq2.2'}), (\ref{eq2.3}) and (\ref{eq2.4'}) that
$z_{m+1-j,j}=0$
for $j=1,\ldots,b$ and that
$z_{i,j}=-z_{m+1-j,m+1-i}$
for $(i,j)\in\{(i,j)\mid j\le b,\ i+j\ge m+2\}$. Hence the condition (\ref{eq2.5}) follows from the condition (\ref{eq2.4'}) and the condition
\begin{equation}
i\ge a+1,\ j\le b,\ i+j\le a+b\Longrightarrow z_{i,j}=0. \label{eq2.6}
\end{equation}

Conversely suppose that $g=g(X,Z)$ satisfies the conditions (\ref{eq2.4'}) and (\ref{eq2.6}). Then $g(V\cap W_2)\subset V$. Write
$$ge_{m+j}=e_{m+j}+\sum_{i=1}^m y_{i,j}e_i$$
for $j=1,\ldots,2d+1$. Then $y_{i,j}=-x_{2d+2-j,m+1-i}$ by (\ref{eq2.0}). Since
$$y_{i,j}=-x_{2d+2-j,m+1-i}=0\mbox{ for }(i,j)\in\{(i,j)\mid i\ge a+1,\ j\le d+1\}$$
by (\ref{eq2.4'}), we also have $g(V\cap W_1)\subset V$. Hence $gV=V$.

Thus it follows from Lemma \ref{lem2.1} that
\begin{align*}
|N_{W_0}V| & =|N_{W_0}|/|\{g\in N_{W_0}\mid gV=V\}|=r^{b(d+1)+b(b-1)/2} \\
& =r^{((n-a)(n-a+1)-d(d+1))/2}.
\end{align*}
On the other hand, we have
$|L_{W_0}V|=[r]_{a+b}/[r]_a[r]_b$
since $L_{W_0}\cong{\rm GL}(W_0)$ and $h(V\cap W_2)$ is the orthogonal subspace of $h(V\cap W_0)$ in $W_2$ for $h\in L_0$. Thus we have the desired formula for $|QV|$.
\end{proof}

Fix a standard maximally isotropic subspace $V$ in $\bbf^{2n+1}$ and let $M_0(V)$ denote the subvariety of $M_0$ consisting of full flags $\mcf: V_1\subset\cdots\subset V_n$ such that $V_n=V$.

For a full flag $\mcf:V_1\subset\cdots\subset V_n$ contained in $M_0(V)$, define
$$a_i=a_i(\mcf)=\dim(V_i\cap W_0),\ c_i=c_i(\mcf)=\dim(V_i\cap W_0^\perp)-\dim(V_i\cap W_0)$$
$$\mand b_i=b_i(\mcf)=\dim V_i-\dim(V_i\cap W_0^\perp)$$
for $i=0,1,\ldots,n$. Define subsets
\begin{align*}
I_{(\alpha)} & =\{\alpha_1,\ldots,\alpha_a\}=\{i\in I\mid a_i=a_{i-1}+1\}, \\
I_{(\lambda)} & =\{\lambda_1,\ldots,\lambda_d\}=\{i\in I\mid c_i=c_{i-1}+1\}, \\
I_{(\beta)} & =\{\beta_1,\ldots,\beta_b\}=\{i\in I\mid b_i=b_{i-1}+1\}
\end{align*}
of $I=\{1,\ldots,n\}$ with $\alpha_1<\cdots<\alpha_a,\ \lambda_1<\cdots<\lambda_d$ and $\beta_1<\cdots<\beta_b$. Then $I=I_{(\alpha)}\sqcup I_{(\lambda)}\sqcup I_{(\beta)}$. Consider the permutation
$$\sigma=\sigma(\mcf):(1\,2\cdots n)\mapsto (\alpha_1\cdots\alpha_a\lambda_1\cdots\lambda_d\beta_1\cdots\beta_b)$$
of $I$. Then the inversion number $\ell(\sigma)$ of $\sigma$ is
$$\ell(\sigma)=|\{(i,j)\mid \alpha_i>\lambda_j\}|+|\{(i,j)\mid \alpha_i>\beta_j\}|+|\{(i,j)\mid \lambda_i>\beta_j\}|.$$

\begin{definition} \ A full flag $\mcf: V_1\subset\cdots\subset V_n=V$ in $M_0(V)$ is called $L_{W_1}$-standard if
\begin{align*}
V_i & =(V_i\cap W_0)\oplus (V_i\cap W_1)\oplus (V_i\cap W_2), \\
V_i\cap W_0 & =\bbf e_1\oplus\cdots\oplus \bbf e_{a_i(\mcf)} \\
\mbox{and}\quad V_i\cap W_2 & =\bbf e_{n+d+2}\oplus\cdots\oplus \bbf e_{n+d+1+b_i(\mcf)}
\end{align*}
for $i=1,\ldots,n$.
\end{definition}

Write $Q_V=Q\cap P_V$ and $(L_{W_1})_V=L_{W_1}\cap P_V$. By Proposition \ref{prop2.3} (ii), we have
$$P_{W_0}\cap P_V=Q_V(L_{W_1})_V.$$

\begin{proposition} \ {\rm (i)} For every full flag $\mcf: V_1\subset\cdots\subset V_n=V$ in $M_0(V)$, there exists a $g\in Q_V$ such that $g\mcf: gV_1\subset\cdots\subset gV_n=V$ is $L_{W_1}$-standard.

{\rm (ii)} Let $\mcf$ and $\mcf'$ be two $L_{W_1}$-standard full flags in $M_0(V)$. If $\mcf'=g\mcf$ for some $g=g_Qg_L\in P_{W_0}\cap P_V=Q_V(L_{W_1})_V$, then $\mcf'=g_L\mcf$.

{\rm (iii)} When $\bbf=\bbf_r$, we have $|Q_V\mcf|=r^{\ell(\sigma(\mcf))}[r]_a[r]_b$.
\label{prop2.5}
\end{proposition}

\begin{proof} (i) Since $V$ is standard, it is written as
$$V=(V\cap W_0)\oplus(V\cap W_1)\oplus(V\cap W_2),$$
with $V\cap W_0=U_{(\alpha)}$ and $V\cap W_2=U_{(\beta)}$. We may moreover assume $V\cap W_1=\bbf e_{m+1}\oplus\cdots\oplus \bbf e_n$ replacing $V$ by $\ell V$ with some $\ell\in L_{W_1}$. By the same reason, we may assume
$$\pi_1(V_{\lambda_i}\cap(W_0\oplus W_1))=\bbf e_{m+1}\oplus\cdots\oplus \bbf e_{m+i}$$
for $i=1,\ldots,d$ where $\pi_1: W_0\oplus W_1\to W_1$ is the projection.

By the action of $L_{W_0}\cap P_V$, we may assume
$$V_i\cap W_0=\bbf e_1\oplus\cdots\oplus \bbf e_{a_i(\mcf)}\mand \pi_2(V_i)=\bbf e_{n+d+2}\oplus\cdots\oplus \bbf e_{n+d+1+b_i(\mcf)}$$
for $i=1,\ldots,n$. We can take vectors
$$w_j=e_{m+j}+\sum_{i=1}^a \widetilde{y}_{i,j}e_i\in (V\cap W_0)\oplus(V\cap W_1)$$
for $j=1,\ldots,d$ such that
$\bbf w_1\oplus\cdots\oplus \bbf w_j\subset V_{\lambda_j}$.
We can also take vectors
$$w'_j=e_{n+d+1+j}+\sum_{i=1}^a \widetilde{z}_{i,j}e_i +\sum_{i=1}^d \widetilde{x}_{i,j}e_{m+i}\in V$$
for $j=1,\ldots,b$ such that
$\bbf w'_1\oplus\cdots\oplus \bbf w'_j\subset V_{\beta_j}$.
Take $g=g(X,Z)\in N_{W_0}\cap P_V$ with $X=\{x_{i,j}\}$ and $Z=\{z_{i,j}\}$ so that
\begin{align*}
x_{i,j} & =\begin{cases} \widetilde{x}_{i,j} & \text{for $i\le d,\ j\le b$}, \\
0 & \text{for $i\ge d+1,\ j\le b$}, \\
0 & \text{for $i\le d+1,\ j\ge b+1$}, \\
\widetilde{y}_{m+1-j,2d+2-i} & \text{for $i\ge d+2,\ j\ge b+1$}, \end{cases} \\
z_{i,j} & =\begin{cases} \widetilde{z}_{i,j} & \text{for $i\le a,\ j\le b$}, \\
0 & \text{if $i\ge a+1$ or $j\ge b+1$}. \end{cases}
\end{align*}
Then we have
$$g^{-1}V_i=(g^{-1}V_i\cap W_0)\oplus (g^{-1}V_i\cap W_1)\oplus (g^{-1}V_i\cap W_2)$$
with
$g^{-1}V_i\cap W_0 =\bbf e_1\oplus\cdots\oplus \bbf e_{a_i(\mcf)}, \ 
g^{-1}V_i\cap W_1 =\bbf e_{m+1}\oplus\cdots\oplus \bbf e_{m+c_i(\mcf)}$ and 
$g^{-1}V_i\cap W_2 =\bbf e_{n+d+2}\oplus\cdots\oplus \bbf e_{n+d+1+b_i(\mcf)}$
for $i=1,\ldots,n$.

(ii) The condition $\mcf'=g\mcf$ implies $a_i(\mcf)=a_i(\mcf')$ and $b_i(\mcf)=b_i(\mcf')$ for $i=1,\ldots,n$. So we have
$$V_i\cap W_0=V'_i\cap W_0\mand V_i\cap W_2=V'_i\cap W_2.$$
We have only to show that $g_L(V_i\cap W_1)=V'_i\cap W_1$ for $i=1,\ldots,n$. Let $v$ be an element of $V_i\cap W_1$. Then we have
$$gv\in V'_i\cap W_0^\perp=(V'_i\cap W_0)\oplus (V'_i\cap W_1).$$
So we can write $gv=v_0+v_1$ with some $v_0\in V'_i\cap W_0$ and $v_1\in V'_i\cap W_1$. Since $g_L v\in W_1$ and since $gv=g_Qg_L v\in g_L v+W_0$, it follows that
$g_L v=v_1$.
Thus we have proved $g_L(V_i\cap W_1)=V'_i\cap W_1$ since $\dim(V_i\cap W_1)=\dim(V'_i\cap W_1)$.

(iii) By (i), we may assume that $\mcf$ is $L_{W_1}$-standard. As in the proof of (i), we may assume
$V_{\lambda_i}\cap W_1=\bbf e_{m+1}\oplus\cdots\oplus \bbf e_{m+i}$
for $i=1,\ldots,d$. Note that
$$gV=V\Longleftrightarrow \mbox{(\ref{eq2.4'}) and (\ref{eq2.5})}$$
for $g=g(X,Z)\in N_{W_0}$.

Suppose $g\mcf=\mcf$. Then since $ge_{m+j}\in V_{\lambda_j}$, we have
$$i\le a,\ \alpha_i>\lambda_j\Longrightarrow y_{i,j}=-x_{2d+2-j,m+1-i}=0$$
for $j=1,\ldots,d$. Since $ge_{n+d+1+j}\in V_{\beta_j}$, we also have
\begin{align*}
i\le d,\ \lambda_i>\beta_j & \Longrightarrow x_{i,j}=0 \\
\mand i\le a,\ \alpha_i>\beta_j & \Longrightarrow z_{i,j}=0
\end{align*}
for $j=1,\ldots,b$. Conversely if $g\in N_{W_0}\cap P_V$ satisfies the above three conditions, then $g\mcf=\mcf$.

Thus we have
$$|(N_{W_0}\cap P_V)\mcf|=|N_{W_0}\cap P_V|/|\{g\in N_{W_0}\cap P_V\mid g\mcf=\mcf\}|=r^{\ell(\sigma(\mcf))}.$$
Clearly $|(L_{W_0}\cap P_V)\mcf|=[r]_a[r]_b$ (the product of the numbers of full flags in $V\cap W_0$ and $V\cap W_2$). So we have
$|Q_V\mcf|=r^{\ell(\sigma(\mcf))}[r]_a[r]_b$.
\end{proof}

\subsection{First reduction}

First apply Proposition \ref{prop2.3} to the case of $d=0$. Noting that $Q=P_{W_0}=P$ in this case, we get the $P$-orbit decomposition
\begin{equation}
M=\bigsqcup_{i=0}^n PU_i \label{eq2.7}
\end{equation}
of $M$. (Of course, we can also deduce this from the Bruhat decomposition of $G$.) Furthermore if $\bbf=\bbf_r$, then
\begin{equation}
|PU_i|=r^{i(i+1)/2}{[r]_n\over [r]_i[r]_{n-i}} \label{eq2.8}
\end{equation}
by Proposition \ref{prop2.3} (iii). By (\ref{eq2.7}), we have a decomposition
$M\times M=\bigsqcup_{i=0}^n\{(gU_0,gU_i)\mid g\in G\}$.
Since the isotropy subgroup of $G$ at $(U_0,U_i)\in M\times M$ is
$R_i=P\cap P_{U_i}$,
we can write
\begin{equation}
M\times M\cong \bigsqcup_{i=0}^n G/R_i. \label{eq2.1}
\end{equation}

By (\ref{eq2.1}), we have only to describe $G$-orbits on $(G/R_i)\times M$ and $(G/R_i)\times M_0$ with respect to the diagonal action of $G$ for $i=0,\ldots,n$. These orbit decompositions are identified with the $R_i$-orbit decompositions of $M$ and $M_0$, respectively.

\subsection{Second reduction}

Consider $R_d$-orbit decompositions of $M$ and $M_0$ for $d=0,\ldots,n$. Since
$U_0\cap U_d=\bbf e_1\oplus\cdots\oplus \bbf e_{n-d}=W_0$,
we have
$Q\subset R_d\subset P_{W_0}$
in the setting of Section 2.1. So we can apply the results in Section 2.1 once more. By Proposition \ref{prop2.3} and Proposition \ref{prop2.5}, the problems are reduced to the $L_{W_1}\cap R_d$-orbit decompositions of $M$ and $M_0$. Note that
$$L_{W_1}\cap R_d=\{g\in L_{W_1}\mid g(W_1\cap U_0)=W_1\cap U_0\mbox{ and }g(W_1\cap U_d)=W_1\cap U_d\}\cong {\rm GL}_d(\bbf)$$
since $(W_1\cap U_0)\cap (W_1\cap U_d)=\{0\}$. So we will consider the case of $d=n$ in the next subsection.

\subsection{The case of $d=n$}

Assume $d=n$. For $A\in {\rm GL}_n(\bbf)$, define
$$h[A]=\bp A & 0 & 0 \\ 0 & 1 & 0 \\ 0 & 0 & J\,{}^tA^{-1}J \ep \mbox{ with } J=J_n=\bp 0 && 1 \\ & \edots & \\ 1 && 0 \ep$$
as in Section 1.1. Then we can write
$R_n=\{h[A]\mid A\in {\rm GL}_n(\bbf)\}\cong {\rm GL}_n(\bbf)$.
Write $H=R_n$ in this subsection. For a maximal isotropic subspace $V$ in $\bbf^{2n+1}$, define
$$c_+=c_+(V)=\dim (V\cap U_0),\ c_-=c_-(V)=\dim (V\cap U_n),\ c_0=c_0(V)=n-c_+-c_-$$
$$\mand \varepsilon(V)=\dim(V+(U_0\oplus U_n))-2n\in \{0,1\}.$$

Let $M(U_{(+)},U_{(-)})$ denote the subvariety of $M$ consisting of $V\in M$ such that
$$V\cap U_0=U_{(+)}=\bbf e_1\oplus\cdots\oplus \bbf e_{c_+}\mbox{ and that } V\cap U_n=U_{(-)}=\bbf e_{n+2}\oplus\cdots\oplus \bbf e_{n+c_-+1}.$$
By the action of $H\cong {\rm GL}_n(\bbf)$, we may assume $V\in M(U_{(+)},U_{(-)})$.
Let $\pi_+:\bbf^{2n+1}\to U_0,\ \pi_-:\bbf^{2n+1}\to U_n$ and $\pi_Z:\bbf^{2n+1}\to Z=\bbf e_{n+1}$ denote the projections with respect to the direct sum decomposition $\bbf^{2n+1}=U_0\oplus U_n\oplus Z$. Write $U_{(0+)}=\bbf e_{c_++1}\oplus\cdots\oplus \bbf e_{c_++c_0}$ and $U_{(0-)}=\bbf e_{n+c_-+2}\oplus\cdots\oplus \bbf e_{n+c_-+c_0+1}$. Then
$$U_{(0)}=U_{(0+)}\oplus U_{(0-)}\oplus Z.$$

\begin{lemma} \ {\rm (i)} ${\rm ker}\,\pi_+|_V=U_{(-)}$ and ${\rm ker}\,\pi_-|_V=U_{(+)}$.

{\rm (ii)} $\pi_+(V)=U_{(+)}\oplus U_{(0+)}$ and $\pi_-(V)=U_{(-)}\oplus U_{(0-)}$.

{\rm (iii)} The inner product $(\ ,\ )$ is nondegenerate on the pair $(U_{(0+)},U_{(0-)})$.
\label{lem2.6}
\end{lemma}

\begin{proof} (i) By the symmetry, we have only to prove the first equality. The kernel of the map $\pi_+|_V:V\to U_0$ is $V\cap (U_n\oplus Z)$. If an element $v\in V$ is written as $v=v_-+z$ with $v_-\in U_n$ and $z\in Z$, then we have
$$0=(v,v)=(v_-+z,v_-+z)=(z,z)$$
and hence $z=0$. Thus we have $V\cap (U_n\oplus Z)=V\cap U_n=U_{(-)}$.

(ii) For $u\in V$ and $v\in U_{(-)}$, we have
$$0=(u,v)=(\pi_+(u)+\pi_-(u)+\pi_Z(u),v)=(\pi_+(u),v).$$
Hence $\pi_+(V)\subset U_{(+)}\oplus U_{(0+)}$. On the other hand, the dimension of $\pi_+(V)$ is $n-\dim U_{(-)}=n-c_-$ by (i). Hence the equality holds. By the symmetry, we have the second formula.

(iii) is clear from the definition.
\end{proof}

\begin{corollary} \ $V=U_{(+)}\oplus U_{(-)}\oplus (V\cap U_{(0)})$.
\end{corollary}

By Lemma \ref{lem2.6} (i) and (ii), we have linear isomorphisms
\begin{align*}
\pi_+|_{V\cap U_{(0)}}: & V\cap U_{(0)}\to U_{(0+)} 
\mand\pi_-|_{V\cap U_{(0)}}: V\cap U_{(0)}\to U_{(0-)}.
\end{align*}
So we can define a linear isomorphism
$$f_V=\pi_-|_{V\cap U_{(0)}}\circ(\pi_+|_{V\cap U_{(0)}})^{-1}: U_{(0+)}\to U_{(0-)}.$$
We also define a linear form $\varphi_V: U_{(0+)}\to \bbf$ by
$$\varphi_V(v)=(e_{n+1},\pi_+|_{V\cap U_{(0)}}^{-1}(v)).$$
By Lemma \ref{lem2.6} (iii) and the above argument, we can define a nondegenerate bilinear form on $U_{(0+)}$ by
$$\langle u,v\rangle_V=(f_V(u),v)$$
for $u,v\in U_{(0+)}$. Define the alternating part and the symmetric part of $\langle \ ,\ \rangle_V$ by
$$\langle u,v \rangle_V^{\rm alt}={1\over 2}(\langle u,v \rangle_V -\langle v,u \rangle_V) 
\mand \langle u,v \rangle_V^{\rm sym}={1\over 2}(\langle u,v \rangle_V +\langle v,u \rangle_V),$$
respectively. Let $u=u_++u_Z+u_-$ and $v=v_++v_Z+v_-$ be elements of $V\cap U_{(0)}$ with $u_+,v_+\in U_{(0+)},\ u_Z,v_Z\in Z$ and $u_-,v_-\in U_{(0-)}$. Then we have
\begin{align*}
0=(u,v) & =(u_++u_Z+u_-,v_++v_Z+v_-) \\
& =(u_+,v_-)+(u_Z,v_Z)+(u_-,v_+) \\
& =(u_+,f_V(v_+))+(u_Z,v_Z)+(f_V(u_+),v_+).
\end{align*}
Hence we have
\begin{equation}
\langle u,v \rangle_V^{\rm sym}={1\over 2}(\langle u,v \rangle_V +\langle v,u \rangle_V)=-{1\over 2}\varphi_V(u)\varphi_V(v) \label{eq2.1'}
\end{equation}
for $u,v\in U_{(0+)}$. In particular, if $u$ or $v$ is in the kernel of $\varphi_V$, then
\begin{equation}
\langle u,v \rangle_V =\langle u,v \rangle_V^{\rm alt}. \label{eq2.2}
\end{equation}

\begin{lemma} \ {\rm (i)} If $c_0(V)=0$, then $\varepsilon(V)=0$.

{\rm (ii)} If $c_0(V)$ is odd, then $\varepsilon(V)=1$.
\label{lem2.11}
\end{lemma}

\begin{proof} (i) If $c_0(V)=0$, then $V=U_{(+)}\oplus U_{(-)}$. Hence $V\subset U_0\oplus U_n$.

(ii) If $\varepsilon(V)=0$, then the bilinear form $\langle\ ,\ \rangle_V$ is alternating by (\ref{eq2.2}). Since it is also nondegenerate, the dimension $c_0(V)=\dim U_{(0+)}$ is even.
\end{proof}

\begin{proposition} \ {\rm (i)} If $c_0$ is even, then $\langle \ ,\ \rangle_V^{\rm alt}$ is nondegenerate on $U_{(0+)}$.

{\rm (ii)} If $c_0$ is odd, then $\langle \ ,\ \rangle_V^{\rm alt}$ is nondegenerate on $\varphi_V^{-1}(0)$.
\label{prop2.7'}
\end{proposition}

\begin{proof} (i) By (\ref{eq2.2}), we may assume that $\varphi_V$ is nontrivial. Suppose that $\langle \ ,\ \rangle_V^{\rm alt}$ is degenerate. Then the subspace
$$Y=\{y\in U_{(0+)}\mid \langle y,v\rangle_V^{\rm alt}=0\mbox{ for all }v\in U_{(0+)}\}$$
of $U_{(0+)}$ is nontrivial and even-dimensional. Take a nonzero element $y$ of $Y\cap \varphi_V^{-1}(0)$. Then for all $v\in U_{(0+)}$, it follows from (\ref{eq2.2}) that
$\langle y,v\rangle_V=\langle y,v\rangle_V^{\rm alt}=0$,
contradicting that $\langle \ ,\ \rangle_V$ is nondegenerate on $U_{(0+)}$.

(ii) By (\ref{eq2.2}), $\langle \ ,\ \rangle_V=\langle \ ,\ \rangle_V^{\rm alt}$ on $\varphi_V^{-1}(0)$. Suppose that it is degenerate on $\varphi_V^{-1}(0)$. Then the subspace
$$Y=\{y\in \varphi_V^{-1}(0)\mid \langle y,v\rangle_V=0\mbox{ for all }v\in \varphi_V^{-1}(0)\}$$
of $\varphi_V^{-1}(0)$ is nontrivial and even-dimensional. Take a $v_0\in U_{(0+)}-\varphi_V^{-1}(0)$ and write
$$Y'=\{y\in Y\mid \langle y,v_0\rangle_V=0\}.$$
Then $Y'$ is nontrivial and
$\langle y,v\rangle_V=0$
for all $v\in U_{(0+)}$, contradicting that $\langle \ ,\ \rangle_V$ is nondegenerate on $U_{(0+)}$.
\end{proof}

\begin{proposition} \ Let $V$ and $V'$ be two elements of $M(U_{(+)},U_{(-)})$. Then the following three conditions are equivalent$:$

{\rm (i)} $V=V'$.

{\rm (ii)} $f_V=f_{V'}$ and $\varphi_V=\varphi_{V'}$.

{\rm (iii)} $\langle\ ,\ \rangle_V^{\rm alt}=\langle\ ,\ \rangle_{V'}^{\rm alt}$ and $\varphi_V=\varphi_{V'}$.
\label{prop2.8}
\end{proposition}

\begin{proof} (ii)$\Longrightarrow$(i). By the direct sum decomposition
$U_{(0)}=U_{(0+)}\oplus Z\oplus U_{(0-)}$,
$V\cap U_{(0)}$ is written as
\begin{align*}
V\cap U_{(0)} & =\{(\pi_+(v),\pi_Z(v),\pi_-(v))\mid v\in V\cap U_{(0)}\} 
=\{(u,\varphi_V(u)e_n,f_V(u))\mid u\in U_{(0+)}\}.
\end{align*}
Hence $V\cap U_{(0)}$ is determined by the two maps $f_V$ and $\varphi_V$.

(iii)$\Longrightarrow$(ii). By (\ref{eq2.1'}), $\langle\ ,\ \rangle_V^{\rm sym}$ is determined by $\varphi_V$. Hence the bilinear form $\langle\ ,\ \rangle_V$ is determined by $\langle\ ,\ \rangle_V^{\rm alt}$ and $\varphi_V$. Since the inner product $(\ ,\ )$ defines a nondegenerate pairing between $U_{(0+)}$ and $U_{(0-)}$, the map $f_V$ is determined by $\langle\ ,\ \rangle_V$.

The implication (i)$\Longrightarrow$(iii) is trivial.
\end{proof}

\begin{proposition} \ {\rm (i)} Let $V$ be an element of $M$. Write $c_\pm=c_\pm(V),\ c_0=c_0(V)$ and $\varepsilon=\varepsilon(V)$. Then
$$V\in \begin{cases} HV(0,0,c_+,c_-)_{\rm odd} & \text{if $c_0$ is odd}, \\
HV(0,0,c_+,c_-)_{\rm even}^0 & \text{if $c_0$ is even and $\varepsilon=0$}, \\
HV(0,0,c_+,c_-)_{\rm even}^1 & \text{if $c_0$ is even and $\varepsilon=1$}.
\end{cases}$$

{\rm (ii)} Let $L_0\cong{\rm GL}_{c_0}(\bbf)$ be as in Section 1.2. Then
$$\{\ell\in L_0\mid \ell V=V\}\cong\begin{cases} 1\times {\rm Sp}_{c_0-1}(\bbf) & \text{if $V=V(0,0,c_+,c_-)_{\rm odd}$}, \\
{\rm Sp}_{c_0}(\bbf) & \text{if $V=V(0,0,c_+,c_-)_{\rm even}^0$}, \\
Q_{c_0} & \text{if $V=V(0,0,c_+,c_-)_{\rm even}^1$}
\end{cases}$$
where $Q_{c_0}=\{g\in {\rm Sp}_{c_0}(\bbf)\mid gv=v\}$ with some nonzero $v\in\bbf^{c_0}$.

{\rm (iii)} If $\bbf=\bbf_r$, then
$$|HV|={[r]_n\over [r]_{c_+}[r]_{c_-}[r]_{c_0}}\psi_{c_0}^\varepsilon(r).$$
\label{prop2.14}
\end{proposition}

\begin{proof} (i) By the action of $H=R_n$, we may assume $V\cap U_0=U_{(+)}$ and $V\cap U_n=U_{(-)}$. By the above arguments, the space $V$ defines an alternating form $\langle\ ,\ \rangle_V^{\rm alt}$ on $U_{(0+)}$ and a linear form $\varphi_V:U_{(0+)}\to \bbf$.

For $V=V(0,0,c_+,c_-)_{\rm even}^\varepsilon$, the form $\langle\ ,\ \rangle_V^{\rm alt}$ is standard: 
$$\langle e_{c_++i},e_{c_++j}\rangle_V^{\rm alt}=\langle e_{c_++i},e_{c_++j}\rangle_0=\begin{cases} \delta_{i,c_0+1-j} & \text{if $i\le c_0/2$}, \\ 
-\delta_{i,c_0+1-j} & \text{if $i>c_0/2$}. \end{cases}$$
The linear form $\varphi_V$ vanishes for $V=V(0,0,c_+,c_-)_{\rm even}^0$ and
$\varphi_V(e_{c_++i})=\delta_{i,c_0}$
for $V=V(0,0,c_+,c_-)_{\rm even}^1$. On the other hand, for $V=V(0,0,c_+,c_-)_{\rm odd}\ (c_0=2c_1-1)$, we have
$$\langle e_{c_++i},e_{c_++j}\rangle_V^{\rm alt}=\langle e_{c_++i},e_{c_++j}\rangle'_0=\begin{cases} \delta_{i,c_0-j} & \text{if $i<c_1$}, \\
0 & \text{if $i=c_1$}, \\ 
-\delta_{i,c_0-j} & \text{if $i>c_1$}, \end{cases}$$
and $\varphi_V(e_{c_++i})=\delta_{i,c_1}$. Note that $\langle\ ,\ \rangle'_0$ is nondegenerate on $W=\varphi_V^{-1}(0)=\bbf e_{c_++1}\oplus\cdots\oplus \bbf e_{c_++c_1-1}\oplus \bbf e_{c_++c_1+1}\oplus\cdots\oplus \bbf e_{c_++c_0}$.

Let $L_0$ be the subgroup of $H$ defined by
$$L_0=\left\{h\left[\bp I_{c_+} & 0 & 0 \\ 0 & A & 0 \\ 0 & 0 & I_{c_-} \ep \right] \Bigm| A\in {\rm GL}_{c_0}(\bbf)\right\}\cong {\rm GL}(U_{(0+)})$$
as in Section 1.2. By Proposition \ref{prop2.8}, we have only to take an $\ell\in L_0$ such that $\langle\ ,\ \rangle_{\ell V}^{\rm alt}$ and $\varphi_{\ell V}$ are standard. Here we note that
\begin{align*}
\langle u,v\rangle_{\ell V}^{\rm alt} & ={1\over 2}((f_{\ell V}(u),v)-(u,f_{\ell V}(v))) \\
& ={1\over 2}(((\ell\circ f_V\circ\ell^{-1})(u),v)-(u,(\ell\circ f_V\circ\ell^{-1})(v))) \\
& ={1\over 2}(((f_V\circ\ell^{-1})(u),\ell^{-1}(v))-(\ell^{-1}(u),(f_V\circ\ell^{-1})(v))) \\
& =\langle \ell^{-1}u,\ell^{-1}v\rangle_V^{\rm alt}.
\end{align*}

First suppose that $c_0$ is even. Then $\langle\ ,\ \rangle_V^{\rm alt}$ is nondegenerate on $U_{(0+)}$ by Proposition \ref{prop2.7'}. So there exists an $\ell\in L_0$ such that $\langle\ ,\ \rangle_{\ell V}^{\rm alt}$ is equal to the standard alternating form $\langle\ ,\ \rangle_0$ on $U_{(0+)}$. If $\varphi_{\ell V}=0$, then we have proved that $\ell V=V(0,0,c_+,c_-)_{\rm even}^0$. Suppose $\varphi_{\ell V}\ne 0$. Then we can take a nonzero element $v_0$ of $U_{(0+)}$ such that
$\varphi_{\ell V}(v)=\langle v_0,v\rangle_0$
for all $v\in U_{(0+)}$. Let $S$ denote the subgroup of $L_0$ defined by
\begin{equation}
S=\{g\in L_0\mid \langle gu,gv\rangle_0=\langle u,v\rangle_0\mbox{ for all }u,v\in U_{(0+)}\}\cong {\rm Sp}_{c_0}(\bbf). \label{eq2.11}
\end{equation}
Then we can take an element $\ell_0$ of $S$ such that $\ell_0 v_0=e_{c_++1}$. We have
$$\varphi_{\ell_0\ell V}(v)=\varphi_{\ell V}(\ell_0^{-1}v)=\langle v_0,\ell_0^{-1}v\rangle_0  =\langle \ell_0 v_0,v\rangle_0 =\langle e_{c_++1},v\rangle_0.$$
Since $\langle e_{c_++1},e_{c_++i}\rangle_0=\delta_{i,c_0}$, we have proved $\ell_0\ell V=V(0,0,c_+,c_-)_{\rm even}^1$.

Next suppose that $c_0$ is odd. Since the bilinear form $\langle\ ,\ \rangle_V$ is nondegenerate on $U_{(0+)}$, we can take a nonzero element $v_0\in U_{(0+)}$ such that
$\varphi_V(v)=\langle v_0,v\rangle_V$
for all $v\in U_{(0+)}$. Since $\langle v_0,v\rangle_V=0$ for all $v\in \varphi_V^{-1}(0)$, it follows that $v_0\notin \varphi_V^{-1}(0)$. Take an element $\ell$ of $L_0$ such that
$$\ell v_0=e_{c_++c_1}\mbox{ and that }\ell\varphi_V^{-1}(0)=W.$$
Then we have
$\langle e_{c_++c_1},v\rangle_{\ell V}^{\rm alt}=0$
for all $v\in U_{(0+)}$ and $\varphi_{\ell V}(e_{c_++i})=\delta_{i,c_1}$. Since $\langle\ ,\ \rangle_{\ell V}^{\rm alt}$ is nondegenerate on $W=\varphi_{\ell V}^{-1}(0)$, we can take an element $\ell_0$ of
$$L'_0=\{\ell\in L_0\mid \ell e_{c_++c_1}=e_{c_++c_1}\mbox{ and }\ell W=W\}\cong {\rm GL}_{c_0-1}(\bbf)$$
such that $\langle\ ,\ \rangle_{\ell_0\ell V}^{\rm alt}=\langle\ ,\rangle'_0$ on $W$. Thus we have $\ell_0\ell V=V(0,0,c_+,c_-)_{\rm odd}$.

(ii) By Proposition \ref{prop2.8}, we have
$$\ell V=V\Longleftrightarrow \langle\ ,\ \rangle^{\rm alt}_{\ell V}=\langle\ ,\ \rangle^{\rm alt}_V\mbox{ and }\varphi_{\ell V}=\varphi_V$$
for $\ell\in L_0$. If $V=V(0,0,c_+,c_-)_{\rm even}^0$, then $\langle\ ,\ \rangle^{\rm alt}_V=\langle\ ,\ \rangle_0$ and $\varphi_V=0$. Hence
$$\ell V=V\Longleftrightarrow \ell\in S\cong {\rm Sp}_{c_0}(\bbf)$$
with the subgroup $S$ of $L_0$ defined in (\ref{eq2.11}). Next suppose that $V=V(0,0,c_+,c_-)_{\rm even}^1$. Then $\langle\ ,\ \rangle^{\rm alt}_V=\langle\ ,\ \rangle_0$ and $\varphi_V(v)=\langle e_{c_++1},v\rangle_0$ for $v\in U_{(0+)}$. Hence
$$\ell V=V\Longleftrightarrow \ell\in S\mbox{ and }\ell e_{c_++1}=e_{c_++1}.$$
Finally suppose that $V=V(0,0,c_+,c_-)_{\rm odd}$. Then $\langle\ ,\ \rangle^{\rm alt}_V=\langle\ ,\ \rangle'_0$ and $\varphi_V(e_{c_++i})=\delta_{i,c_1}$ for $i=1,\ldots,c_0$. Hence
$$\ell V=V\Longleftrightarrow \langle \ell u,\ell v\rangle'_0=\langle u,v\rangle'_0\mbox{ for }u,v\in U_{(0+)}\mbox{ and }\varphi_{\ell V}=\varphi_V \Longleftrightarrow \ell\in S'$$
where $S'=\{\ell\in L_0\mid \ell(e_{c_++c_0})=e_{c_++c_0},\ \ell(W)=W\mbox{ and }\langle \ell u,\ell v\rangle'_0=\langle u,v\rangle'_0\mbox{ for }u,v\in W\}\cong 1\times {\rm Sp}_{c_0-1}(\bbf)$.

(iii) If we fix $c_+$ and $c_-$, then we have $[r]_n/[r]_{c_+}[r]_{c_-}[r]_{c_0}$ choices of the pair $(V\cap U_0,V\cap U_n)$. On the other hand, every element of $M(U_{(+)},U_{(-)})$ is contained in the $L_0$-orbit of $V=V(0,0,c_+,c_-)_{\rm odd},\ V(0,0,c_+,c_-)_{\rm even}^0$ or $V(0,0,c_+,c_-)_{\rm even}^1$ as is proved in (i). Since $|L_0V|=\psi_{c_0}^\varepsilon(r)$ by (ii), we have the desired formula for $HV$.
\end{proof}

Let $P_H$ be the parabolic subgroup of $H=R_n$ defined by
\begin{align*}
P_H & =\{g\in H\mid gU_{(+)}=U_{(+)}\mbox{ and }g(U_{(+)}\oplus U_{(0+)})=U_{(+)}\oplus U_{(0+)}\} \\
& =\{g\in H\mid gU_{(+)}=U_{(+)}\mbox{ and }gU_{(-)}=U_{(-)}\}.
\end{align*}
Then the unipotent radical $N_H$ of $P_H$ is written as
$$N_H=\left\{h\left[\bp I_{c_+} & * & * \\ 0 & I_{c_0} & * \\ 0 & 0 & I_{c_-} \ep \right] \right\}$$
and 
$$L=\left\{h\left[\bp A & 0 & 0 \\ 0 & B & 0 \\ 0 & 0 & C \ep \right] \Bigm| A\in {\rm GL}_{c_+}(\bbf),\ B\in {\rm GL}_{c_0}(\bbf),\ C\in {\rm GL}_{c_-}(\bbf)\right\}$$
is a Levi subgroup of $P_H$. Let $\mathcal{M}(p,q;\bbf)$ denote the space of $p\times q$ matrices with entries in $\bbf$. Then every element of $N_H$ is written as
$$n(A,B,C)=h\left[\bp I_{c_+} & A & C \\ 0 & I_{c_0} & B \\ 0 & 0 & I_{c_-} \ep \right]$$
with $A\in\mathcal{M}(c_+,c_0;\bbf),\ B\in\mathcal{M}(c_0,c_-;\bbf)$ and $C\in\mathcal{M}(c_+,c_-;\bbf)$.

Fix a $V\in M(U_{(+)},U_{(-)})$. By Definition \ref{def1.6}, a full flag $\mcf: V_1\subset\cdots\subset V_n$ is called standard if
$$V_n=V\mand V_i=(V_i\cap U_{(+-)})\oplus (V_i\cap U_{(0)})$$
for all $i=1,\ldots,n$. Here we write $U_{(+-)}=U_{(+)}\oplus U_{(-)}$.

\begin{lemma} \ {\rm (i)} For every $g\in N_H$, there exists a linear map
$$f_g: V\cap U_{(0)}\to U_{(+-)}$$
such that
$gv=v+f_g(v)$
for all $v\in V\cap U_{(0)}$.

{\rm (ii)} The map $g\mapsto f_g$ is a surjection of $N_H$ onto the space of $\bbf$-linear maps ${\rm Hom}(V\cap U_{(0)},U_{(+-)})$.

\label{lem2.9}
\end{lemma}

\begin{proof} (i) For $g=n(A,B,C)$, we have
$g\pi_+(v)=\pi_+(v)+A\pi_+(v)$
for $v\in V\cap U_{(0)}$. Here the matrix $A$ is naturally identified with the linear map $A:U_{(0+)}\to U_{(+)}$. Since
$$J\bp I_{c_+} & 0 & 0 \\ {}^tA & I_{c_0} & 0 \\ {}^tC & {}^tB & I_{c_-} \ep^{-1}J=\bp I_{c_-} & -J_{c_-}\,{}^tBJ_{c_0} & * \\ 0 & I_{c_0} & -J_{c_0}\,{}^tAJ_{c_+} \\ 0 & 0 & I_{c_+} \ep,$$
we have
$g\pi_-(v)=\pi_-(v)-J_{c_-}\,{}^tBJ_{c_0}\pi_-(v)$.
Hence
$$gv=g\pi_+(v)+g\pi_-(v)+g\pi_Z(v)=v+A\pi_+(v)-J_{c_-}\,{}^tBJ_{c_0}\pi_-(v).$$
So we can write $f_g(v)=A\pi_+(v)-J_{c_-}\,{}^tBJ_{c_0}\pi_-(v)$.

(ii) For every element $f\in {\rm Hom}(V\cap U_{(0)},U_{(+-)})$, $f(v)$ is written as
$$f(v)=A\pi_+(v)-J_{c_-}\,{}^tBJ_{c_0}\pi_-(v)$$
with some $A\in\mathcal{M}(c_+,c_0;\bbf)$ and $B\in\mathcal{M}(c_0,c_-;\bbf)$.
\end{proof}

Let $\mcf: V_1\subset\cdots\subset V_n$ be a full flag in $\bbf^{2n+1}$ such that $V_n=V$. Define subsets $I_1$ and $I_2$ of $I=\{1,\ldots,n\}$ by
\begin{align*}
I_1 & =I_1(\mcf)=\{\gamma_1,\ldots,\gamma_c\}=\{i\in I\mid V_i\cap U_{(+-)} \supsetneqq V_{i-1}\cap U_{(+-)}\}, \\
I_2 & =I_2(\mcf)=\{\delta_1,\ldots,\delta_{c_0}\}=\{i\in I\mid V_i\cap U_{(+-)} = V_{i-1}\cap U_{(+-)}\}
\end{align*}
where $c=c_++c_-$ and $\gamma_1<\cdots<\gamma_c,\ \delta_1<\cdots<\delta_{c_0}$. Let $\tau(\mcf)$ denote the permutation
$$\tau(\mcf): (1\,2\cdots n)\mapsto (\gamma_1\cdots\gamma_c\delta_1\cdots\delta_{c_0})$$
of $I$ and $\ell(\tau(\mcf))$ the inversion number. Then
$$\ell(\tau(\mcf))=(c_++c_-)c_0-\sum_{i\in I_2(\mcf)} \dim(V_i\cap U_{(+-)}).$$

\begin{proposition} \ {\rm (i)} Let $\mcf: V_1\subset\cdots\subset V_n$ be a full flag in $\bbf^{2n+1}$ such that $V_n=V$. Then there exists a $g\in N_H$ such that $g\mcf: gV_1\subset\cdots\subset gV_n$ is standard.

{\rm (ii)} Let $\mcf: V_1\subset\cdots\subset V_n$ and $\mcf': V'_1\subset\cdots\subset V'_n$ be two standard full flags in $\bbf^{2n+1}$. Suppose $\mcf'=g\mcf$ with some $g=g_Ng_L\in N_HL_H=P_H$. Then $\mcf'=g_L\mcf$.

{\rm (iii)} If $\bbf=\bbf_r$, then $|N_H\mcf|=r^{\ell(\tau(\mcf))}$.
\label{prop2.14'}
\end{proposition}

\begin{proof} (i) Take elements $v_i\in V_i-V_{i-1}$ for $i\in I_2$ and define a subspace
$W=\bigoplus_{i\in I_2}\bbf v_i$
of $V$. Then we have $V=U_{(+-)}\oplus W$ and
$$V_i=(V_i\cap U_{(+-)})\oplus(V_i\cap W)$$
for $i=1,\ldots,n$.

The subspace $W$ is written as $W=\{v+f(v)\mid v\in V\cap U_{(0)}\}$ with some $f\in {\rm Hom}(V\cap U_{(0)},U_{(+-)})$. By Lemma \ref{lem2.9}, there exists a $g\in N_H$ such that $gW=V\cap U_{(0)}$. Hence the flag $g\mcf$ is standard.

(ii) Suppose $V_i=(V_i\cap U_{(+-)})\oplus(V_i\cap U_{(0)}),\ V'_i=(V'_i\cap U_{(+-)})\oplus(V'_i\cap U_{(0)})$ and $gV_i=V'_i$ with some $g=g_Ng_L\in N_HL_H=P_H$. For an element $v\in V_i\cap U_{(0)}$, write
$gv=v'_0+v'_1$
with $v'_0\in V'_i\cap U_{(+-)}$ and $v'_1\in V'_i\cap U_{(0)}$. Since $g_L v=g_N^{-1}gv\in U_{(+-)}+v'_1$, it follows that $g_L v=v'_1$. Hence
$$g_L (V_i\cap U_{(0)})=V'_i\cap U_{(0)}.$$
Since $g(V_i\cap U_{(+-)})=V'_i\cap U_{(+-)}$ and since $g_N$ acts trivially on $U_{(+-)}$, it also follows that $g_L (V_i\cap U_{(+-)})=V'_i\cap U_{(+-)}$.

(iii) We may assume that $\mcf$ is standard. Let $g$ be an element of $N_H$. Then
$$g\mcf=\mcf\Longleftrightarrow f_g(V_i\cap U_{(0)})\subset V_i\cap U_{(+-)}\mbox{ for all }i\in I_2.$$
Take $v_i\in (V_i\cap U_{(0)})-(V_{i-1}\cap U_{(0)})$ for $i\in I_2$. Then $\{v_i\mid i\in I_2\}$ is a basis of $V\cap U_{(0)}$. The above condition is equivalent to
$$f_g(v_i)\in V_i\cap U_{(+-)}\mbox{ for all }i\in I_2.$$
The number of such $f_g$ is
$r^{\sum_{i\in I_2} \dim(V_i\cap U_{(+-)})}$.
On the other hand, the number of elements in ${\rm Hom}(V\cap U_{(0)},U_{(+-)})$ is $r^{(c_++c_-)c_0}$. Hence $|N_H\mcf|=r^{\ell(\tau(\mcf))}$.
\end{proof}

\subsection{Proof of Theorem \ref{th1.2}, Theorem \ref{th1.4}, Proposition \ref{prop1.5} and Theorem \ref{th1.6}}
\indent

Theorem \ref{th1.2} follows from Proposition \ref{prop2.3} (i), Lemma \ref{lem2.11} and Proposition \ref{prop2.14} (i). Proposition \ref{prop1.5} is proved in Proposition \ref{prop2.14}.

\bigskip
\noindent {\it Proof of Theorem \ref{th1.4}.} By Theorem \ref{th1.2}, we may assume that $t=(U_0,U_d,V)\ (d=n-a-b)$ with $V=V(a,b,c_+,c_-)_{\rm odd},\ V=V(a,b,c_+,c_-)_{\rm even}^0$ or $V=V(a,b,c_+,c_-)_{\rm even}^1$. Since
\begin{equation}
|Gt|=|\{(gU_0,gU_d)\mid g\in G\}||R_dV|=|M||PU_d||R_dV|, \label{eq2.12}
\end{equation}
we have only to compute $|PU_d|$ and $|R_dV|$. By (\ref{eq2.8}), we have
\begin{equation}
|PU_d|=r^{d(d+1)/2}{[r]_n\over [r]_d[r]_{n-d}}. \label{eq2.14}
\end{equation}
On the other hand, it follows from Proposition \ref{prop2.3} (ii) and (iii) that
\begin{equation}
|R_dV|=|Q(R_d\cap L_{W_1})V|=r^{((n-a)(n-a+1)-d(d+1))/2}{[r]_{n-d}\over [r]_a[r]_b}|(R_d\cap L_{W_1})V|. \label{eq2.15}
\end{equation}
By Proposition \ref{prop2.14} (iii), we have
\begin{equation}
|(R_d\cap L_{W_1})V|={[r]_d\over [r]_{c_+}[r]_{c_-}[r]_{c_0}}\psi_{c_0}^\varepsilon(r). \label{eq2.16}
\end{equation}
So the assertion follows from (\ref{eq2.12}), (\ref{eq2.14}), (\ref{eq2.15}) and (\ref{eq2.16}).
\hfill$\square$

\bigskip
Let $U_d$ and $V$ be as above. Write $H=L_{W_1}\cap R_d\cong {\rm GL}_d(\bbf)$. Then $R_d=P\cap P_{U_d}$ is decomposed as $R_d=QH$. By Proposition \ref{prop2.5} (ii),
$$R(t)=P\cap P_{U_d}\cap P_V=R_d\cap P_V=Q_V(H\cap P_V)$$
where $Q_V=Q\cap P_V$. Define a parabolic subgroup
\begin{align*}
P_H & =\{g\in H\mid gU_{(+)}=U_{(+)},\ gU_{(-)}=U_{(-)}\} \\
& =\{g\in H\mid gU_{(+)}=U_{(+)},\ g(U_{(+)}\oplus U_{(0+)})=U_{(+)}\oplus U_{(0+)}\}
\end{align*}
of $H$. Then we have a Levi decomposition $P_H=N_HL$ where $L$ is defined in Section 1.3. Since $N_H\subset H\cap P_V\subset P_H$, we can write
$H\cap P_V=N_HL_V$
with $L_V=L\cap P_V$. Thus we have:

\begin{lemma} \ We have a bijection
$$Q_V\times N_H\times L_V\ni (g,h,\ell)\mapsto gh\ell\in R(t)=P\cap P_{U_d}\cap P_V.$$
\label{lem2.13}
\end{lemma}

\noindent {\it Proof of Theorem \ref{th1.6}.} (i) Let $\mcf: V_1\subset\cdots\subset V_n$ be a full flag in $M_0(V)$. By Proposition \ref{prop2.5} (i), there exists a $g\in Q_V$ such that
$$gV_i=(gV_i\cap W_0)\oplus (gV_i\cap W_1)\oplus (gV_i\cap W_2),$$
$gV_i\cap W_0=\bbf e_1\oplus\cdots\oplus \bbf e_{a_i(\mcf)}$ and that $gV_i\cap W_2=\bbf e_{n+d+2}\oplus\cdots\oplus \bbf e_{n+d+1+b_i(\mcf)}$ for $i=1,\ldots,n$.

Consider the full flag
$gV_{\lambda_1}\cap W_1\subset\cdots\subset gV_{\lambda_d}\cap W_1$
in $W_1$. By Proposition \ref{prop2.14'} (i), there exists an $h\in N_H$ such that
$$hgV_i\cap W_1=(hgV_i\cap (U_{(+)}\oplus U_{(-)}))\oplus (hgV_i\cap U_{(0)}).$$
for $i=1,\ldots,n$. So the flag $hg\mcf$ is standard.

(ii) Let $\mcf$ and $\mcf'$ be two standard full flags in $M_0(V)$ such that $g\mcf=\mcf'$ for some $g\in R(t)$. By Lemma \ref{lem2.13}, we can write $g=g_Qg_Ng_L$ with $g_Q\in Q_V,\ g_N\in N_H$ and $g_L\in L_V$. Since $g_Ng_L\in L_{W_1}$, it follows from Proposition \ref{prop2.5} (ii) that $g_Ng_L\mcf=\mcf'$. By Proposition \ref{prop2.14'} (ii), we also have $g_L\mcf=\mcf'$ as desired.

(iii) Let $\mcf: V_1\subset\cdots\subset V_n$ be a standard full flag in $M_0(V)$. Considering the full flag
$V_{\lambda_1}\cap W_1\subset\cdots\subset V_{\lambda_d}\cap W_1$
in $W_1$, we define a permutation
$$\tau'=\tau'(\mcf): (\lambda_1\cdots\lambda_d)\mapsto (\gamma_1\cdots\gamma_c\delta_1\cdots\delta_{c_0})$$
of the subset $\{\lambda_1,\ldots,\lambda_d\}$ in $\{1,\ldots,n\}$. Let $\ell(\tau')$ be the inversion number
$\ell(\tau')=|\{(\gamma_i,\delta_j)\mid \gamma_i>\delta_j\}|.$
Then we have $\ell(\tau)=\ell(\sigma)+\ell(\tau')$ where
\begin{align*}
\sigma & : (1\,2\cdots n)\mapsto (\alpha_1\cdots\alpha_a\lambda_a\cdots\lambda_d\beta_1\cdots\beta_b) \\
\mand \tau & : (1\,2\cdots n)\mapsto (\alpha_1\cdots \alpha_a\gamma_1\cdots\gamma_c\delta_1\cdots\delta_{c_0}\beta_1\cdots\beta_b)
\end{align*}
are as in Section 2.4 and Section 1.3.

By Proposition \ref{prop2.14'} (ii) and (iii), we have
$|N_HL_V\mcf|=r^{\ell(\tau')}|L_V\mcf|$.
On the other hand, by Proposition \ref{prop2.5} (ii) and (iii), we have
$|Q_VN_HL_V\mcf|=[r]_a[r]_br^{\ell(\sigma)}|N_HL_V\mcf|$.
Hence
$|R(t)\mcf|=[r]_a[r]_br^{\ell(\tau)}|L_V\mcf|$
by Lemma \ref{lem2.13}.
\hfill$\square$

\section{Orbits on ${\rm GL}_n(\bbf)/B$}

\subsection{Preliminaries}

Let $\bbf$ be an arbitrary field. Let
$$V_1\subset V_2\subset\cdots\subset V_{n-1}\subset \bbf^n\mand W_1\subset W_2\subset\cdots\subset W_{n-1}\subset \bbf^n$$
be two full flags in $\bbf^n$. (Write $V_0=W_0=\{0\}$ and $V_n=W_n=\bbf^n$.) Define
$d_{i,j}=\dim(V_i\cap W_j)$
for $i,j=0,1,2,\ldots,n$ and
$c_{i,j}=d_{i,j}-d_{i-1,j}-d_{i,j-1}+d_{i-1,j-1}$
for $i,j=1,2,\ldots,n$.

\begin{proposition} \ The $n\times n$ matrix $\{c_{i,j}\}$ is a permutation matrix.
\label{prop3.1}
\end{proposition}

\begin{proof} Since $V_{i-1}\cap W_{j-1}=(V_{i-1}\cap W_j)\cap (V_i\cap W_{j-1})$ and since
$V_i\cap W_j\supset (V_{i-1}\cap W_j)+(V_i\cap W_{j-1})$,
we have $c_{i,j}\ge 0$. On the other hand, $\sum_{i=1}^n c_{i,j}=d_{n,j}-d_{0,j}-d_{n,j-1}+d_{0,j-1}=j-(j-1)=1$ and  $\sum_{j=1}^n c_{i,j}=d_{i,n}-d_{i,0}-d_{i-1,n}+d_{i,0}=i-(i-1)=1$. Hence $\{c_{i,j}\}$ is a permutation matrix.
\end{proof}

We also have the following by the same arguments as above.

\begin{proposition} \ The following four conditions are equivalent$:$

{\rm (i)} $c_{i,j}=1$.

{\rm (ii)} $d_{i,j}-1=d_{i-1,j}=d_{i,j-1}=d_{i-1,j-1}$.

{\rm (iii)} $V_i\cap W_j\supsetneqq V_{i-1}\cap W_j=V_i\cap W_{j-1}=V_{i-1}\cap W_{j-1}$.

{\rm (iv)} $V_i\cap W_j\supsetneqq (V_{i-1}\cap W_j)+(V_i\cap W_{j-1})$.
\label{prop3.2}
\end{proposition} 

\begin{remark} \ (Bruhat decomposition of $G={\rm GL}_n(\bbf)$) Let $V_1\subset V_2\subset\cdots\subset V_{n-1}$ be the canonical full flag defined by
$V_i=\bbf e_1\oplus\cdots\oplus \bbf e_i$
for $i=1,\ldots,n-1$ with the canonical basis $e_1,\ldots,e_n$ of $\bbf^n$. The subgroup $B$ of $G$ defined by
$$B=\{g\in G\mid gV_i=V_i\mbox{ for }i=1,\ldots,n-1\}=\{\mbox{upper triangular matrices in }G\}$$
is a Borel subgroup of $G$.

Let $W_1\subset W_2\subset\cdots\subset W_{n-1}$ be an arbitrary full flag in $\bbf^n$. By Proposition {\ref{prop3.1}}, we can define a permutation $i=i(j)$ of $\{1,2,\ldots,n\}$ determined by
$i=i(j)\Longleftrightarrow c_{i,j}=1$.
Suppose $i=i(j)$. Then by Proposition {\ref{prop3.2}}, we can take vectors $v_i\in V_i\cap W_j$ for $i=1,\ldots,n$ (for $j=1,\ldots,n$) such that
$$v_i\notin V_{i-1}\cap W_j=V_i\cap W_{j-1}=V_{i-1}\cap W_{j-1}.$$
It follows that $v_1,\ldots,v_n$ is a basis of $\bbf^n$ such that
$V_i=\bbf v_1\oplus\cdots\oplus \bbf v_i$
for $i=1,\ldots,n$. It also follows that $v_{i(1)},\ldots,v_{i(n)}$ is a basis of $\bbf^n$ such that
$W_j=\bbf v_{i(1)}\oplus\cdots\oplus \bbf v_{i(j)}$
for $j=1,\ldots,n$. Define $n\times n$ matrices
$$g=(v_1v_2\cdots v_n)\mand w=\{c_{i,j}\}=(e_{i(1)}e_{i(2)}\cdots e_{i(n)}).$$
Then $g\in B$ and
$gwV_j=W_j$
for $j=1,\ldots,n$. Thus we have proved
$G=\bigsqcup_{w\in \mathcal{W}} BwB$
where $\mathcal{W}$ is the subgroup of $G$ consisting of all the permutation matrices.
\end{remark}

\subsection{${\rm Sp}_{2n}$-orbits on ${\rm GL}_{2n}/B$}

Let $\langle\ ,\ \rangle$ denote the alternating form on $\bbf^{2n}$ defined by
$$\langle e_i,e_j\rangle=\begin{cases} \delta_{i,2n+1-j} & \text{for $i\le n$}, \\
-\delta_{i,2n+1-j} & \text{for $i>n$}.
\end{cases}$$
Define a subgroup $H=\{g\in G\mid \langle gu,gv\rangle=\langle u,v\rangle \mbox{ for all }u,v\in\bbf^{2n}\}\cong {\rm Sp}_{2n}(\bbf)$ of $G={\rm GL}_{2n}(\bbf)$. Let
$V_1\subset V_2\subset\cdots\subset V_{2n-1}$
be a full flag in $\bbf^{2n}$. Then there corresponds another (``decreasing'') full flag
$V_1^\perp\supset V_2^\perp\supset\cdots\supset V_{2n-1}^\perp$
in $\bbf^{2n}$ defined by
$$V_i^{\perp}=\{u\in\bbf^{2n}\mid \langle u,v\rangle =0\mbox{ for all }v\in V_i\}.$$
Define $d_{i,j}=\dim(V_i\cap V_j^\perp)$ and $c_{i,j}=d_{i,j-1}-d_{i,j}-d_{i-1,j-1}+d_{i-1,j}$. Then we have the following two propositions by Proposition \ref{prop3.1} and Proposition \ref{prop3.2}.

\begin{proposition} \ $\{c_{i,j}\}_{i,j=1}^{2n}$ is a permutation matrix.
\end{proposition}

\begin{proposition} \ The following four conditions are equivalent$:$

{\rm (i)} $c_{i,j}=1$.

{\rm (ii)} $d_{i,j-1}-1=d_{i,j}=d_{i-1,j-1}=d_{i-1,j}$.

{\rm (iii)} $V_i\cap V_{j-1}^\perp\supsetneqq V_i\cap V_j^\perp=V_{i-1}\cap V_{j-1}^\perp=V_{i-1}\cap V_j^\perp$.

{\rm (iv)} $V_i\cap V_{j-1}^\perp\supsetneqq (V_i\cap V_j^\perp)+(V_{i-1}\cap V_{j-1}^\perp)$
\label{prop3.4}
\end{proposition}

Since the orthogonal space of $V_i\cap V_j^\perp$ is $V_i^\perp+V_j$, we have
$\dim(V_i^\perp+V_j)=2n-d_{i,j}$.
So we have
\begin{align*}
d_{j,i} & =\dim(V_j\cap V_i^\perp) 
=\dim V_j+\dim V_i^\perp-\dim(V_j+V_i^\perp) \\
& =j+(2n-i)-(2n-d_{i,j})
=d_{i,j}+j-i
\end{align*}
and hence $c_{i,j}=c_{j,i}$.

\begin{lemma} \ Suppose that $c_{i,j}=1$. Then we have $\langle u,v\rangle \ne 0$ for all $u\in (V_i\cap V_{j-1}^\perp)-(V_{i-1}\cap V_{j-1}^\perp)$ and $v\in (V_j\cap V_{i-1}^\perp)-(V_{j-1}\cap V_{i-1}^\perp)$.
\label{lem3.6'}
\end{lemma}

\begin{proof} Suppose $u\in (V_i\cap V_{j-1}^\perp)-(V_{i-1}\cap V_{j-1}^\perp)$ and $v\in (V_j\cap V_{i-1}^\perp)-(V_{j-1}\cap V_{i-1}^\perp)$. Then
$$\langle u,V_{j-1}\rangle =\{0\}\mand V_j=V_{j-1}\oplus \bbf v.$$
If $\langle u,v\rangle =0$, then
$$\langle u,V_j\rangle =\langle u,V_{j-1}\oplus \bbf v\rangle =\langle u,V_{j-1}\rangle =\{0\}$$
and hence $u\in V_i\cap V_j^\perp=V_{i-1}\cap V_{j-1}^\perp$ by Proposition \ref{prop3.4} (iii). But this contradicts the choice of $u$.
\end{proof}

\noindent {\it Proof of Proposition \ref{prop1.9}}. Suppose $c_{i,i}=1$. Then we can take an element $v\in (V_i\cap V_{i-1}^\perp)-(V_{i-1}\cap V_{i-1}^\perp)$ by Proposition \ref{prop3.4} (iii). By Lemma \ref{lem3.6'}, we have $\langle v,v\rangle\ne 0$. But this contradicts that $\langle\ ,\ \rangle$ is alternating.
\hfill$\square$

\bigskip
\noindent {\it Proof of Proposition \ref{prop1.10}}. (i) We will prove this by induction on $n$. Take a pair $(i,j)$ such that $i<j$ and that $c_{i,j}=1$. By Lemma \ref{lem3.6'}, we can take a $v_i\in V_i\cap V_{j-1}^\perp$ and a $v_j\in V_j\cap V_{i-1}^\perp$ such that $\langle v_i,v_j\rangle=1$. Put
$U=\bbf v_i\oplus\bbf v_j$. 
Then we have a direct sum decomposition
$\bbf^{2n}=U\oplus U^\perp$.
If $k\le i-1$, then $V_k\subset U^\perp$. If $i\le k\le j-1$, then
$$V_k=\bbf v_i\oplus (V_k\cap U^\perp)=(V_k\cap U)\oplus (V_k\cap U^\perp).$$
If $k\ge j$, then $V_k\supset U$. Hence for every $k=0,\ldots,2n$, we have
$V_k=(V_k\cap U)\oplus (V_k\cap U^\perp)$.
We also have
$V_\ell^\perp=(V_\ell^\perp\cap U)\oplus (V_\ell^\perp\cap U^\perp)$
and
$$V_k\cap V_\ell^\perp =(V_k\cap V_\ell^\perp\cap U)\oplus (V_k\cap V_\ell^\perp\cap U^\perp)$$
for $k,\ell=0,\ldots,2n$.

We can consider the full flag
$V_1\cap U^\perp\subset\cdots\subset V_{2n-1}\cap U^\perp$
in $U^\perp$ neglecting the two coincidences $V_{i-1}\cap U^\perp=V_i\cap U^\perp$ and $V_{j-1}\cap U^\perp=V_j\cap U^\perp$. Define
$d'_{k,\ell}=\dim V_k\cap U_\ell\cap U^\perp$, $d''_{k,\ell}=\dim V_k\cap U_\ell\cap U$, 
$c'_{k,\ell}=d'_{k,\ell-1}-d'_{k,\ell}-d'_{k-1,\ell-1}+d'_{k-1,\ell}$ and $c''_{k,\ell}=d''_{k,\ell-1}-d''_{k,\ell}-d''_{k-1,\ell-1}+d''_{k-1,\ell}$
for $k,\ell\in I=\{1,\ldots,2n\}$. Then
$d_{k,\ell}=d'_{k,\ell}+d''_{k,\ell}$ and $c_{k,\ell}=c'_{k,\ell}+c''_{k,\ell}$. 
For $k\in I-\{i,j\}$, we see that
 $V_{k-1}\cap U=V_k\cap U$. 
So we have $c''_{k,\ell}=0$ and
 $c_{k,\ell}=c'_{k,\ell}$ 
for $k,\ell\in I-\{i,j\}$.

By the assumption of induction, we can take a basis
$$v_1,\ldots,v_{i-1},v_{i+1},\ldots,v_{j-1},v_{j+1},\ldots,v_{2n}$$
of $U^\perp$ such that
$$V_k\cap U^\perp=\bigoplus_{\ell\in\{1,\ldots,k\}-\{i,j\}} \bbf v_\ell$$
and that $\langle v_k,v_\ell\rangle=c_{k,\ell}$ for $k,\ell\in I-\{i,j\}$. Thus the basis
$v_1,\ldots,v_{2n}$
of $\bbf^{2n}$ satisfies the desired properties. (Remark: We may take $(i,j)=(i_1,j_1)=(1,j_1)$ in the above proof. But we will need such a general argument as above in the proof of Proposition \ref{prop1.13}.)

(ii) Since $\bbf=\bbf_r$ consists of $r$ elements, we have $r-1$ choices of $v_1$ in $V_1=V_1\cap V_{j_1-1}^\perp$. Fix $v_1$. Then we have $r^{j_1-1}$ choices of $v_{j_1-1}$ in $V_{j_1}=V_{j_1}\cap V_0^\perp$ such that $\langle v_1,v_{j_1}\rangle =1$. Write
$$\ell_2=|\{i\mid 2<i\mbox{ and }\sigma(2)>\sigma(i)\}|.$$
Then we have $j_1=\sigma(2)=\ell_2+2$.

Fix $v_1$ and $v_{j_1}$. Then the subspaces $U=\bbf v_1\oplus \bbf v_{j_1}$ and $U^\perp$ in (i) are determined. Next we take $v_{i_2}\in V_{i_2}\cap U^\perp\cong\bbf$. So we have $r-1$ choices of $v_{i_2}$. Fix $v_{i_2}$. We see that
$$\dim(V_{j_2}\cap U^\perp)=\begin{cases} j_2-2 & \text{if $j_1<j_2$}, \\ j_2-1 & \text{if $j_1>j_2$.} \end{cases}$$
On the other hand, if we write
$$\ell_4=|\{i\mid 4<i\mbox{ and }\sigma(4)>\sigma(i)\}|,$$
then we have
$$\ell_4=\begin{cases} j_2-4 & \text{if $j_1<j_2$}, \\ j_2-3 & \text{if $j_1>j_2$.} \end{cases}$$
Hence we have $r^{\ell_4+1}$ choices of $v_{j_2}\in V_{j_2}\cap U^\perp$ such that $\langle v_{i_2},v_{j_2}\rangle=1$.

Repeating this procedure, we have $r-1$ choices of $v_{i_k}$ and $r^{\ell_{2k}+1}$ choices of $v_{j_k}$ if we fix $v_1,v_{j_1},\ldots,v_{i_{k-1}},v_{j_{k-1}}$. Here
$$\ell_{2k}=|\{i\mid 2k<i\mbox{ and }\sigma(2k)>\sigma(i)\}|$$
for $k=1,\ldots,n$. Since $\ell(\sigma)=\ell_2+\ell_4+\cdots+\ell_{2n}$, we have $(r-1)^nr^{n+\ell(\sigma)}$ choices of the bases $v_1,\ldots,v_{2n}$.
\hfill$\square$

\subsection{$Q_{2n}$-orbits on ${\rm GL}_{2n}/B$}

Let $H\cong {\rm Sp}_{2n}(\bbf)$ be as in the previous subsection. Let $Q_{2n}$ denote the subgroup of $H$ defined by
$Q_{2n}=\{g\in H\mid ge_{2n}=e_{2n}\}$.
Then $Q_{2n}$ stabilizes the hyperplane
$W=\bbf e_2\oplus\cdots\oplus \bbf e_{2n}=(\bbf e_{2n})^\perp$
in $\bbf^{2n}$.

Let $V_1\subset\cdots\subset V_{2n-1}$ be an arbitrary full flag in $\bbf^{2n}$. 
Let $S$ denote the subset of $I\times I$ defined by
$S=\{(i,j)\mid V_i\cap V_{j-1}^\perp\not\subset W\}$.
Define 
$$S_0=\{(i,j)\in S\mid V_i\cap V_j^\perp\subset W\mbox{ and }V_{i-1}\cap V_{j-1}^\perp\subset W\}.$$

\begin{lemma} \ Suppose $(i,j)\in S_0$. Then we have$:$

{\rm (i)} $c_{i,j}=1$.

{\rm (ii)} $V_i\cap V_{j-1}^\perp\cap W=V_{i-1}\cap V_{j-1}^\perp=V_i\cap V_j^\perp=V_{i-1}\cap V_j^\perp$.
\label{lem3.12}
\end{lemma}

\begin{proof} (i) If $(i,j)\in S_0$, then
$(V_i\cap V_j^\perp)+(V_{i-1}\cap V_{j-1}^\perp)\subset W$.
So we have
$(V_i\cap V_j^\perp)+(V_{i-1}\cap V_{j-1}^\perp)\subsetneqq V_i\cap V_{j-1}^\perp$
and hence $c_{i,j}=1$ by Proposition \ref{prop3.4}.

(ii) By (i), it follows from Proposition \ref{prop3.4} that
$$V_i\cap V_{j-1}^\perp\supsetneqq V_{i-1}\cap V_{j-1}^\perp=V_i\cap V_j^\perp=V_{i-1}\cap V_j^\perp.$$
So the assertion is clear.
\end{proof}

It is clear that
\begin{equation}
(i,j)\in S\mbox{ and }i\le i',\ j\ge j'\Longrightarrow (i',j')\in S. \label{eq3.2}
\end{equation}
So the subset $S_0$ determines $S$ by
$$(i,j)\in S\Longleftrightarrow i\ge i_0\mbox{ and }j\le j_0\mbox{ for some }(i_0,j_0)\in S_0.$$
For example, if $n=2$ and $S_0=\{(2,1),(4,3)\}$, then
$$S=\{(2,1),(3,1),(4,1),(4,2),(4,3)\}.$$

\bigskip
\noindent {\it Proof of Proposition \ref{prop1.12}}. By Lemma \ref{lem3.12}, we may assume that $x_1<x_2<\cdots<x_s$. If $i<j$ and $y_j<y_i$, then it follows from $(x_i,y_i)\in S_0$ that $(x_j,y_j+1)\in S$ by (\ref{eq3.2}). But this contradicts that $(x_j,y_j)\in S_0$. Thus we have
$$y_1\le\cdots\le y_s.$$
By Lemma \ref{lem3.12}, the numbers $x_1,\ldots,x_s,y_1,\ldots,y_s$ are distinct.
\hfill$\square$

\bigskip
Write $I_{(A)}=I-\{x_1,\ldots,x_s,y_1,\ldots,y_s\}$.

\bigskip
\noindent {\it Proof of Proposition \ref{prop1.13}}. \ (i) We will prove this by induction on $n$. First assume that $s<n$. Then we can take a pair $(i,j)$ in $I_{(A)}$ such that $i<j$ and that $c_{i,j}=1$. By Lemma \ref{lem3.6'}, we can take a $v_i\in V_i\cap V_{j-1}^\perp$ and a $v_j\in V_j\cap V_{i-1}^\perp$ such that $\langle v_i,v_j\rangle =1$. If $(i,j)\notin S$, then
$$v_i\in V_i\cap V_{j-1}^\perp\subset W.$$
On the other hand, if $(i,j)\in S$, then there exists an $(x_t,y_t)\in S_0$ such that $i>x_t$ and $j<y_t$. By Lemma \ref{lem3.12}, there exists a $v\in V_{x_t}\cap V_{y_t-1}$ such that $v\notin W$. Since $x_t<i$, we have $\langle v,v_j\rangle =0$. If $v_i\notin W$, then we can replace $v_i$ by $v_i+\alpha v\in V_i\cap V_{j-1}^\perp\cap W$ with some $\alpha\in\bbf^\times$ since $V_{x_t}\subset V_i$ and $V_{y_t}^\perp\subset V_{j-1}^\perp$. Thus we may assume that $v_i\in W$. In the same way, we may also assume that $v_j\in W$.

Put $U=\bbf v_i\oplus\bbf v_j$ and consider the direct sum decomposition
$\bbf ^{2n}=U\oplus U^\perp$
as in the proof of Proposition \ref{prop1.10}. Since $U\subset W$, $e_{2n}$ is contained in $U^\perp$. So we may assume that we have chosen a basis
$$v_1,\ldots,v_{i-1},v_{i+1},\ldots,v_{j-1},v_{j+1},\ldots,v_{2n}$$
of $U^\perp$ satisfying the conditions (a), (b) and (c) for $U^\perp$ by the assumption of induction. Then the basis $v_1,\ldots,v_{2n}$ of $\bbf ^{2n}$ is a desired one.

Finally we consider the case of $s=n$. By Proposition \ref{prop1.10}, we can take a basis $v_1,\ldots,v_{2n}$ of $\bbf^{2n}$ such that
$V_i=\bbf v_1\oplus\cdots\oplus \bbf v_i$
for $i=1,\ldots,2n$ and that
$\langle v_i,v_j\rangle =c_{i,j}$
for $i<j$. By Lemma \ref{lem3.12}, $v_{x_1},\ldots,v_{x_s}$ are not contained in $W$. So we can normalize them so that
$$\langle v_{x_1},e_{2n}\rangle =\cdots =\langle v_{x_s},e_{2n}\rangle =1.$$
We can also normalize $v_{y_1},\ldots,v_{y_s}$ so that $\langle v_{x_t},v_{y_t}\rangle =\varepsilon_t$ where
$$\varepsilon_t=\begin{cases} 1 & \text{if $x_t<y_t$,} \\ -1 & \text{if $x_t>y_t$.} \end{cases}$$
If $y_t<x_t$, then $(y_t,x_t)\notin S$. So we have $v_{y_t}\in W$. On the other hand, if $y_t>x_t$ and $v_{y_t}\notin W$, then we can replace $v_{y_t}$ by $v_{y_t}+\alpha v_{x_t}\in W$ with some $\alpha\in\bbf^\times$. Thus the basis $v_1,\ldots,v_{2n}$ of $\bbf^{2n}$ satisfies the desired properties.

(ii) By Proposition \ref{prop1.10}, we have $(r-1)^nr^{n+\ell(\sigma)}$ choices of the bases without the condition (c). By the first condition
$$v_i\in W\mbox{ for }i\ne x_1,\ldots,x_s$$
in (c), the number of choices of $v_i$ becomes $1/r$ of the number without the condition (c) for each $i$ such that
\begin{equation}
(i,j)\notin S-S_0\mbox{ with }c_{i,j}=1 \label{eq3.2''}
\end{equation}
as in the proof of (i). Since there are $m=m(\mcf)$ indices $i$ satisfying the condition (\ref{eq3.2''}) by the definition, we divide the number by $r^m$. On the other hand, the second condition
$$\langle v_{x_1},e_{2n}\rangle=\cdots =\langle v_{x_s},e_{2n}\rangle=1$$
in (c) is the condition on the length of vectors $v_{x_1},\ldots,v_{x_s}$. So we divide the number by $(r-1)^s$. Thus we have $(r-1)^{n-s}r^{n+\ell(\sigma)-m}$ choices of the bases.
\hfill$\square$

\bigskip
\noindent {\it Proof of Theorem \ref{th1.14}}. (i) For each partition $I=I_{(A)}\sqcup I_{(X)}\sqcup I_{(Y)}$ and each $\{c_{i,j}\}\in C(I_{(A)})$, we construct a ``standard'' basis $u_1,\ldots, u_{2n}$ of $\bbf^{2n}$ as follows. We can take a unique subsequence $i_1<\cdots<i_{n-s}$ in $I_{(A)}$ such that $c_{i_1,j_1}=\cdots=c_{i_{n-s},j_{n-s}}=1$ with some $j_1,\ldots,j_{n-s}\in I_{(A)}$ and that $i_t<j_t$ for $t=1,\ldots,n-s$. Define
\begin{align*}
u_{i_1} & =e_{s+1},\ \ldots,\ u_{j_{n-s}}=e_n, \\
u_{j_1} & =e_{2n-s},\ \ldots,\ u_{j_{n-s}}=e_{n+1}, \\
u_{x_1} & =e_1+e_2,\ u_{x_2}=e_1+e_3,\ \ldots,\ u_{x_{s-1}}=e_1+e_s,\ u_{x_s}=e_1, \\
u_{y_1} & =\varepsilon_1e_{2n-1},\ u_{y_2}=\varepsilon_2e_{2n-2},\ \ldots,\ u_{y_{s-1}}=\varepsilon_{s-1}e_{2n-s+1}, \\
u_{y_s} & =\varepsilon_s(e_{2n}-e_{2n-1}-\cdots-e_{2n-s+1}).
\end{align*}
Then the basis vectors $u_1,\ldots, u_{2n}$ satisfy the properties:
\begin{equation}
\langle u_i,u_j\rangle =c_{i,j}\mbox{ for }i<j, \label{eq3.2'}
\end{equation}
\begin{equation}
u_i\in W\mbox{ for }i\ne x_1,\ldots,x_s \label{eq3.3}
\end{equation}
and
\begin{equation}
\langle u_{x_1},e_{2n}\rangle =\cdots=\langle u_{x_s},e_{2n}\rangle =1. \label{eq3.4'}
\end{equation}

Let $v_1,\ldots,v_{2n}$ be the basis of $\bbf^{2n}$ given in Proposition \ref{prop1.13} (i).
Let $g$ be the element of ${\rm GL}_{2n}(\bbf)$ defined by
$gv_i=u_i\mbox{ for }i=1,\ldots,2n$.
By (\ref{eq3.2'}) and Proposition \ref{prop1.13} (i) (b), $g$ is an element of $H\cong {\rm Sp}_{2n}(\bbf)$. By (\ref{eq3.3}), (\ref{eq3.4'}) and Proposition \ref{prop1.13} (i) (c), we have
\begin{align*}
W & =\bbf (u_{x_1}-u_{x_2}) \oplus\cdots\oplus \bbf (u_{x_{s-1}}-u_{x_s})\oplus \bigoplus_{i\notin\{x_1,\ldots,x_s\}}\bbf u_i \\
& =\bbf (v_{x_1}-v_{x_2}) \oplus\cdots\oplus \bbf (v_{x_{s-1}}-v_{x_s})\oplus \bigoplus_{i\notin\{x_1,\ldots,x_s\}}\bbf v_i.
\end{align*}
So the element $g$ stabilizes $W$. Hence we have $ge_{2n}=\beta e_{2n}$ with some $\beta\in\bbf^\times$. Moreover we have
$$\beta=\langle u_{x_1},\beta e_{2n}\rangle =\langle gv_{x_1},ge_{2n}\rangle =\langle v_{x_1},e_{2n}\rangle =1.$$
Thus we have proved $g\in Q_{2n}$. Since the flag $gV_1\subset\cdots\subset gV_n$ is written as $gV_i=\bbf u_1\oplus\cdots\oplus \bbf u_i\ (i=1,\ldots,2n-1)$ using the standard basis $u_1,\ldots,u_n$, we have proved (i).

(ii) is clear and (iii) follows from Proposition \ref{prop1.13} (ii).
\hfill$\square$

\subsection{Proof of Theorem \ref{th1.19}}

\noindent {\it Proof of Lemma \ref{lem1.18}}. Let $\xi(2k,s)$ be the number of the words with $2k$ letters consisting of
$$\mbox{a,b},\ldots,\mbox{A,B},\ldots,\mbox{X,Y}$$
with $|I_X|=|I_Y|=s$. There are $k-s$ pairs of $\mbox{a,b},\ldots,$ or $\mbox{A,B},\ldots$ and we can assign a signature $+1$ or $-1$ to each pair according as they are small or capital. So we have
$$\xi(2k,s)=2^{k-s}(2k-2s-1)(2k-2s-3)\cdots 1{(2k)!\over (s!)^2(2k-2s)!}={(2k)!\over (s!)^2(k-s)!}.$$
Thus we have the desired formula for $\xi(2k)=\sum_{s=0}^k \xi(2k,s)$. We can prove $\xi(2k-1)=\sum_{s=1}^k (2k-1)!/s!(s-1)!(k-s)!$ in the same way.
\hfill$\square$

\bigskip
\noindent {\it Proof of Theorem \ref{th1.19}}. (i) We can divide every word $w$ into the subword $w_1$ consisting of $\alpha,\beta,+,-$ and the subword $w_2$ consisting of the other letters. There are $4^{n-k}$ choices of $w_1$ and $\xi(k)$ choices of $w_2$ if the length of $w_2$ is $k$. Considering the number of partitions of $w$ into two subwords with $n-k$ letters and $k$ letters, we get the desired formula
$$|\Delta G\backslash M\times M\times M_0|= \sum_{k=0}^n 4^{n-k}{n\choose k}\xi(k).$$

The proof of (ii) is similar because we have only to consider words without $\alpha$ and $\beta$.
\hfill$\square$

\subsection{Proof of Theorem \ref{th1.20}}

\noindent {\it Proof of Theorem \ref{th1.20}}. (i) is proved easily.

(ii) By the same arguments as in the proof of Proposition \ref{prop2.3} (iii), we have
$|(N_{W_0}\cap G')V|=r^{bd+b(b-1)/2}=r^{((n-a)(n-a-1)-d(d-1))/2}$
and hence
\begin{equation}
|(Q\cap G')V|=r^{((n-a)(n-a-1)-d(d-1))/2}{[r]_{n-d}\over [r]_a[r]_b}. \label{eq3.4}
\end{equation}
Noting that $R_d\cap L_{W_1}\subset G'$, we have
$$|(R_d\cap G')V|=|(Q\cap G')(R_d\cap L_{W_1})V|=r^{((n-a)(n-a-1)-d(d-1))/2}{[r]_{n-d}\over [r]_a[r]_b}|(R_d\cap L_{W_1})V|.$$
On the other hand, we also have
$|(P\cap G')U_d|=r^{d(d-1)/2}[r]_n/ ([r]_d[r]_{n-d})$
by (\ref{eq3.4}). Combining with (\ref{eq2.16}), we get the desired formula
$$|G't|=|M^0||(P\cap G')U_d||(R_d\cap G')V|=|M^0|{r^{(n-a)(n-a-1)/2}[r]_n\psi_{c_0}^0(r)\over [r]_a[r]_b[r]_{c_+}[r]_{c_0}[r]_{c_-}}.\qquad\square$$

\section{Appendix}

Let $\bbf$ be an arbitrary field and let $\bbf^n=U_+\oplus U_-$ be the direct sum decomposition of $\bbf^n$ with
$U_+=\bbf e_1\oplus\cdots\oplus \bbf e_{m_+}$ and $U_-=\bbf e_{m_++1}\oplus\cdots\oplus \bbf e_n\ (n=m_++m_-)$.
Let $H$ denote the subgroup of $G={\rm GL}_n(\bbf)$ defined by
$$H=\{g\in G\mid gU_+=U_+,\ gU_-=U_-\}.$$

In this appendix, we will give a proof of the $H$-orbit decomposition on the full flag variety of $G={\rm GL}_n(\bbf)$.

Let $\pi_+:\bbf^n\to U_+$ and $\pi_-:\bbf^n\to U_-$ denote the projections with respect to the direct sum decomposition $\bbf^n=U_+\oplus U_-$. For a full flag $\mcf:V_1\subset\cdots\subset V_{n-1}$ in $\bbf^n$, define
$$d_{i,j}^+ =\dim(\pi_+(V_i)\cap V_j),\quad d_{i,j}^-=\dim(\pi_-(V_j)\cap V_i)$$
for $i,j=0,\ldots,n$ and
$$c_{i,j}^\pm =d_{i,j}^\pm-d_{i,j-1}^\pm-d_{i-1,j}^\pm+d_{i-1,j-1}^\pm,\quad c_{i,j}=c_{i,j}^++c_{i,j}^-$$
for $i,j=1,\ldots,n$.

\begin{lemma} \ {\rm (i)} $\{c_{i,j}\}$ is a permutation matrix.

{\rm (ii)} $i>j\Longrightarrow c_{i,j}^+=0$ and $i<j\Longrightarrow c_{i,j}^-=0$.

{\rm (iii)} $c_{i,j}=c_{j,i}$.
\end{lemma}

\begin{proof} (i) As in the proof of Proposition \ref{prop3.1}, we have $c_{i,j}^+\ge 0$ and $c_{i,j}^-\ge 0$. On the other hand, we have
\begin{align*}
\sum_{j=1}^n c_{i,j}^+ & =d_{i,n}^+-d_{i,0}^++d_{i-1,n}^+-d_{i-1,0}^+=\dim\pi_+(V_i)-\dim\pi_+(V_{i-1}) \\
\mand \sum_{j=1}^n c_{i,j}^- & =d_{i,n}^--d_{i,0}^-+d_{i-1,n}^--d_{i-1,0}^-=\dim(V_i\cap U_-)-\dim(V_{i-1}\cap U_-).
\end{align*}
Hence $\sum_{j=1}^n c_{i,j}=\dim V_i-\dim V_{i-1}=1$. In the same way, we also have $\sum_{i=1}^n c_{i,j}=1$. So the matrix $\{c_{i,j}\}$ is a permutation matrix.

(ii) If $i>j$, then $\pi_+(V_{i-1})\cap V_j=\pi_+(V_{i-1})\cap V_j\cap U_+=V_j\cap U_+$. In the same way, we have $\pi_+(V_i)\cap V_j=V_j\cap U_+$. Hence $c_{i,j}^+=0$. The second formula is similar.

(iii) Suppose $i<j$. Then it follows from (ii) that
\begin{align*}
c_{i,j}=c_{i,j}^+ & =-\dim(\pi_+(V_i)+V_j)+\dim(\pi_+(V_i)+V_{j-1}) \\
& \quad +\dim(\pi_+(V_{i-1})+V_j)-\dim(\pi_+(V_{i-1})+V_{j-1})
\end{align*}
since $d_{i,j}^+=\dim \pi_+(V_i)+\dim V_j-\dim(\pi_+(V_i)+V_j)$ for $i,j=0,\ldots,n$. It also follows from (ii) that
\begin{align*}
c_{j,i}=c_{j,i}^- & =-\dim(\pi_-(V_i)+V_j)+\dim(\pi_-(V_i)+V_{j-1}) \\
& \quad +\dim(\pi_-(V_{i-1})+V_j)-\dim(\pi_-(V_{i-1})+V_{j-1}).
\end{align*}
Since $\pi_+(V_k)+V_\ell=\pi_-(V_k)+V_\ell$ for $k\le\ell$, we have $c_{i,j}=c_{j,i}$.
\end{proof}

Define subsets
\begin{align*}
I_{(+)} & =\{k_1^+,\ldots,k_{m_+-s}^+\}=\{i\in I\mid c_{i,i}^+=1\}, \\
I_{(-)} & =\{k_1^-,\ldots,k_{m_--s}^-\}=\{i\in I\mid c_{i,i}^-=1\}, \\
I_{(1)} & =\{i_1,\ldots,i_s\}=\{i\in I\mid c_{i,j}=1\mbox{ for some }j>i\}, \\
I_{(2)} & =\{j_1,\ldots,j_s\}=\{j\in I\mid c_{i,j}=1\mbox{ for some }i<j\}
\end{align*}
of $I=\{1,\ldots,n\}$ with $k_1^+<\cdots<k_{m_+-s}^+,\ k_1^-<\cdots<k_{m_--s}^-,\ j_1<\cdots<j_s$ and
$$c_{i_t,j_t}=1\mbox{ for }t=1,\ldots,s.$$
Then $I=I_{(+)}\sqcup I_{(-)}\sqcup I_{(1)}\sqcup I_{(2)}$. Write $I_{(+)}\sqcup I_{(-)}=\{k_1,\ldots,k_{n-2s}\}$ with $k_1<\cdots<k_{n-2s}$. Define a permutation
$$\sigma: (1\,2\cdots n)\mapsto (i_s\cdots i_1k_1\cdots k_{n-2s}j_1\cdots j_s)$$
of $I$ and the inversion number $\ell(\sigma)$.  (Remark: Let $\tau$ denote the permutaion corresponding to the matrix $\{c_{i,j}\}$. Then we can prove $\ell(\tau)=s(n-s)-2\ell(\sigma)$.)

\begin{proposition} \ For every full flag $\mcf: V_1\subset\cdots\subset V_{n-1}$ in $\bbf^n$, define $c_{i,j}=c_{i,j}^++c_{i,j}^-$ as above.

{\rm (i)} We can take a basis $v_1,\ldots,v_n$ of $\bbf^n$ such that

\indent\indent {\rm (a)} $V_i=\bbf v_1\oplus\cdots\oplus \bbf v_i$ for $i=1,\ldots, n$.

\indent\indent {\rm (b)} $c_{i,i}^+=1\Longrightarrow v_i\in U_+,\ c_{i,i}^-=1\Longrightarrow v_i\in U_-$ and
$$i<j,\ c_{i,j}=1\Longrightarrow v_i\notin U_+\cup U_-,\ v_j=\pi_+(v_i).$$

{\rm (ii)} Define a basis $u_1,\ldots,u_n$ of $\bbf^n$ by
\begin{align*}
u_{k_t^+} & =e_t \qquad \mbox{ for }t=1,\ldots,m_+-s, \\
u_{k_t^-} & =e_{m_++t} \qquad \mbox{ for }t=1,\ldots,m_--s, \\
u_{i_t} & =e_{m_+-s+t}+e_{n-s+t} \qquad \mbox{ for }t=1,\ldots,s, \\
u_{j_t} & =e_{m_+-s+t} \qquad \mbox{ for }t=1,\ldots,s
\end{align*}
and define $g\in G$ by $gv_i=u_i$ for $i=1,\ldots,n$. Then $g\in H$.

{\rm (iii)} If $\bbf=\bbf_r$, then the number $\mathcal{N}(\mcf)$ of the bases satisfying the conditions in {\rm (i)} is
$$(r-1)^{n-s}r^{((m_+-s)(m_+-s-1)+(m_--s)(m_--s-1))/2+\ell(\sigma)}.$$
\label{prop4.3}
\end{proposition}

\begin{proof} (i) Suppose $c_{i,j}^+=1$. Then we can take a $v\in \pi_+(V_i)\cap V_j$ such that $v\notin \pi_+(V_i)\cap V_{j-1}=\pi_+(V_{i-1})\cap V_j$. If $i=j$, then $v_i=v\in U_+$ satisfies $V_i=V_{i-1}\oplus \bbf v_i$. Suppose $i<j$ and take a $v_i\in V_i$ so that $\pi_+(v_i)=v$. Then $V_i=V_{i-1}\oplus \bbf v_i$ and $v_i\notin U_-$. If $v_i\in U_+$, then $v=v_i\in V_i\subset V_{j-1}$ a contradiction. Hence $v_i\notin U_+$. Take $v_j=v=\pi_+(v_i)$. Then $V_j=V_{j-1}\oplus \bbf v_j$.

If $c_{i,i}^-=1$, then we take a $v_i\in \pi_-(V_i)\cap V_i$ such that $v_i\notin \pi_-(V_i)\cap V_{i-1}=\pi_-(V_{i-1})\cap V_i$. Then $v_i\in U_-$ and $V_i=V_{i-1}\oplus \bbf v_i$.

(ii) is clear from (i).

(iii) We will prove this by induction on $n$. First suppose $c_{n,n}^+=1$. Then $V_{n-1}=(V_{n-1}\cap U_+)\oplus U_-$ and so the number of the bases $v_1,\ldots,v_{n-1}$ of $V_{n-1}$ satisfying the conditions in (i) is
$$(r-1)^{n-s-1}r^{((m_+-s-1)(m_+-s-2)+(m_--s)(m_--s-1))/2+\ell(\sigma')}.$$
by the assumption of induction. Here $\sigma'$ is the permutation
$$\sigma': (1\,2\cdots n-1)\mapsto (i_s\cdots i_1k_1\cdots k_{n-2s-1}j_1\cdots j_s).$$
For each basis $v_1,\ldots,v_{n-1}$ of $V_{n-1}$, we have $(r-1)r^{m_+-1}$ choices of $v_n\in U_+-(U_+\cap V_{n-1})$. Since
$$(m_+-s)(m_+-s-1)/2=(m_+-s-1)(m_+-s-2)/2+(m_+-s-1)$$
and since $\ell(\sigma)=\ell(\sigma')+s$, we have
$$\mathcal{N}(\mcf)=(r-1)^{n-s}r^{((m_+-s)(m_+-s-1)+(m_--s)(m_--s-1))/2+\ell(\sigma)}.$$
The case of $c_{n,n}^-=1$ is similar.

So we may assume $c_{p,n}=c_{p,n}^+=1$ with some $p<n$. Consider the subspace $W=(V_{n-1}\cap U_+)\oplus(V_{n-1}\cap U_-)$ of $\bbf^n$. Then the number of the bases $v_1,\ldots,v_{p-1},v_{p+1},\ldots,v_{n-1}$ of $W$ satisfying the conditions in (i) is
$$(r-1)^{n-s-1}r^{((m_+-s)(m_+-s-1)+(m_--s)(m_--s-1))/2+\ell(\sigma')}$$
by the assumption of induction where $\sigma'$ is the permutation
$$\sigma': (1\,2\cdots p-1\,p+1\cdots n-1)\mapsto (i_{s-1}\cdots i_1k_1\cdots k_{n-2s}j_1\cdots j_{s-1})$$
of $\{1\,\ldots, p-1,p+1,\ldots, n-1\}$. (Note that $p=i_s$ and $n=j_s$.) For each basis $v_1,\ldots,v_{p-1},v_{p+1},\ldots,v_{n-1}$ of $W$, we have $(r-1)r^{p-1}$ choices of $v_p\in V_p-V_{p-1}$.
Since $\ell(\sigma)=\ell(\sigma')+p-1$, we have the desired formula for $\mathcal{N}(\mcf)$ noting that $v_n=\pi_+(v_p)$.
\end{proof}

By this proposition, we can express orbits by ``$+-$ab-symbols'' as in Fig.1 (c.f. \cite{MO}). We can easily count the number of orbits:

\begin{corollary} \ $\displaystyle{|H\backslash M|=\sum_{s=0}^{\min(m_+,m_-)} {n!\over 2^ss!(m_+-s)!(m_--s)!}}$.
\end{corollary}

\begin{corollary} \ For each full flag $\mcf$ in $\bbf^n$, define $s$ and $\sigma$ as above. If $\bbf=\bbf_r$, then
$$|H\mcf|=(r-1)^sr^{s(n-s-1)-\ell(\sigma)}[r]_{m_+}[r]_{m_-}.$$
\end{corollary}

\begin{proof} Note that
\begin{align*}
|H| & =(r^{m_+}-1)(r^{m_+}-r)\cdots(r^{m_+}-r^{m_+-1})(r^{m_-}-1)(r^{m_-}-r)\cdots(r^{m_-}-r^{m_--1}) \\
& =(r-1)^nr^{(m_+(m_+-1)+m_-(m_--1))/2}[r]_{m_+}[r]_{m_-}.
\end{align*}
Since $(m_+(m_+-1)+m_-(m_--1))/2-((m_+-s)(m_+-s-1)+(m_--s)(m_--s-1))/2=s(n-s-1)$, we have
$$|H\mcf|=|H|/\mathcal{N}(\mcf)=(r-1)^sr^{s(n-s-1)-\ell(\sigma)}[r]_{m_+}[r]_{m_-}.$$
\end{proof}

\end{document}